\documentclass[a4paper,12pt]{amsart}
\usepackage{amsmath}
\usepackage{amssymb}
\usepackage{amscd}
\usepackage{mathrsfs}
\usepackage[dvipdfmx]{graphicx}
\everymath{\displaystyle}
\usepackage[usenames]{color}
\usepackage[mathlines,pagewise,displaymath]{lineno}

\theoremstyle{plain} 
\newtheorem{theorem}{\indent\sc Theorem}[section]
\newtheorem{lemma}[theorem]{\indent\sc Lemma}
\newtheorem{corollary}[theorem]{\indent\sc Corollary}
\newtheorem{proposition}[theorem]{\indent\sc Proposition}

\theoremstyle{definition} 
\newtheorem{definition}[theorem]{\indent\sc Definition}
\newtheorem{remark}[theorem]{\indent\sc Remark}
\newtheorem{example}[theorem]{\indent\sc Example}

\newtheorem{problem}[theorem]{\indent\sc Problem}

\begin{document}

\pagestyle{plain}
\thispagestyle{plain}

\title[Local criteria for non embeddability of Levi-flat manifolds]
{Local criteria for non embeddability of Levi-flat manifolds}

\author[Takayuki KOIKE and Noboru OGAWA]{Takayuki KOIKE$^{1}$ and Noboru OGAWA$^{2}$}
\address{ 
$^{1}$ Department of Mathematics \\
Graduate School of Science \\
Kyoto University \\
Kyoto 606-8502 \\
Japan 
}
\email{tkoike@math.kyoto-u.ac.jp}
\address{
$^{2}$ Department of Mathematics \\
Tokai University  \\
4-1-1 Kitakaname, Hiratsuka-shi,\\
Kanagawa 259-1292 \\
Japan
}
\email{nogawa@tsc.u-tokai.ac.jp}

\begin{abstract}
We give local criteria for smooth non-embeddability of Levi-flat manifolds. 
For this purpose, we pose an analogue of Ueda theory 
on the neighborhood structure of hypersurfaces 
in complex manifolds with topologically trivial normal bundles. 
\end{abstract}

\maketitle

\section{Introduction}

Our interest is a Levi-flat \textit{$C^{\infty}$-smooth} embedding problem: 
Which Levi-flat manifold appear in a complex manifold 
as a Levi-flat $C^{\infty}$-smooth real hypersurface? 
In this paper, we give a local criterion for non-embeddability. 
To approach this problem, we develop an obstruction theory of 
the existence of the ``global" holomorphic defining functions of complex hypersurfaces, 
which is a generalization of (a part of) \textit{the Ueda Theory}. 

Let $M$ be a smooth manifold of dimension $2n+1$ and 
$\mathcal{F}$
a smooth foliation in $M$ of real codimension $1$ with a transversely smooth, 
leafwise complex structure $J_{\mathcal{F}}$. 
The triple $(M,\mathcal{F}, J_{\mathcal{F}})$ is called a Levi-flat manifold. 
From the foliation point of views, 
there are two significant studies of Levi-flat manifolds:
one is the existence and the classification of (abstruct) Levi-flat manifolds. 
The other is the Levi-flat embedding problem. 
In the former direction, 
Meerssman-Verjovsky \cite{MV} dealt with the moduli problem of Levi-flat manifolds under suitable conditions. 
In this paper, we deal with the latter one. 
A real hypersurface foliated by complex immersed hypersurfaces in a complex manifold 
is called a Levi-flat hypersurface.
Basic examples come from real codimension 1 invariant sets of holomorphic foliations. 
Levi-flat hypersurfaces were studied 
in both fields of several complex analysis and holomorphic foliations. 
One of the origin is in the works for the Levi problem due to Grauert 
(See \cite{Oh1}, \cite{Oh2} and references therein). 
On the other hand, 
the study has been developing in connection with the exceptional minimal set conjecture 
(\cite{CLS},\cite{BLM} and \cite{C}). 
The Levi-flat non-embeddablity problem in $\mathbb{C}P^{n}$ is still open for $n=2$.  
Note that this was affirmatively solved for $n \geq 3$, 
that is, it admits no $C^{\omega}$ Levi-flat embeddings (See \cite{LN}, and for the generalization, 
see \cite{Siu}, \cite{Br}, \cite{Oh3}, \cite{Der}, \cite{AB} and references therein).

In such a background, the embedding problem was studied from 
the perspective of 
the topology and dynamics of Levi-flat foliations. 
There are simple topological obstructions. 
For example, if a compact Levi-flat $3$-manifold is realized in a K\"{a}hler surface, then the foliation is taut. 
In particular, Reeb components can not be done (But it is realized in a Hopf surface \cite{Ne}). 
Barrett \cite{B} showed that a Reeb foliation on $S^3$ 
can not be realized in \textit{any} complex surfaces as a Levi-flat $C^{\infty}$-smooth hypersurface. 
Moreover if a Levi-flat $3$-manifold has a torus leaf 
along which the holonomy group is generated by a $C^{\infty}$-flat contracting holonomy, 
then it admits no Levi-flat $C^{\infty}$-embeddings. 
Barrett and Inaba \cite{BI} gave a topological restriction of 
Levi-flat $C^{\infty}$-hypersurfaces in complex surfaces. 
They also constructed a $C^{\infty}$ (but not $C^{\omega}$) 
Levi-flat immersion into $\mathbb{C}^2$. 
 Della Sala \cite{D} showed a higher dimensional analogue of Barrett's theorem. 
Such a non-embeddablity result depends on the theorem of Ueda \cite{U}, 
which we will review in \S \ref{section:review_ueda_theory}. 
Hence, they needed to require the compactness of a leaf. 
Recently, the first author posed an analogue of Ueda's theory \cite{K}. 
In the present paper, 
we refine the argument and apply 
this results 
to the Levi-flat embedding problem. 
The main result is the following: 

\begin{theorem}\label{thm:main}
Let $(M, \mathcal{F},J_{\mathcal{F}})$ be a Levi-flat $5$-manifold, 
$L$ a leaf of $\mathcal{F}$ and 
$C$ an elliptic curve as a complex submanifold of $L$. 
Assume that 
there exists a neighborhood $\mathcal{U}$ of $C$ in $M$ such that 
$(i)$ the holonomy group $\mathcal{H}(\mathcal{U}\cap L)$ is isomorphic to $\mathbb{Z}$, generated by 
a $C^{\infty}$-flat contracting holonomy along an element of $\pi_{1}(C)$, 
$(ii)$ there exists a $C^\infty$-retraction $p\colon \mathcal{U}\to \mathcal{U}\cap L$ 
whose restriction on each leaf is locally biholomorphic and surjective, 
and $(iii)$ there exists a holomorphic function $f$ defined on $\mathcal{U}\cap L$ 
such that $\{f=0\}=C$. 
Then $(M, \mathcal{F},J_{\mathcal{F}})$ does not admit 
a Levi-flat $C^\infty$-embedding into any complex $3$-manifold. 
\end{theorem}

We emphasize that, in Theorem \ref{thm:main}, the compactness assumption is required only for $C$ 
and thus we can apply our result to foliations without compact leaves. 
Thus we obtain a {\it local} criterion for embeddings in the sense of an arbitrary small neighborhood of $C$. 
For example, we can deduce from 
Theorem \ref{thm:main}
that the Levi-flat manifold $(S^3, \mathcal{F}_{\rm Reeb},J_{\mathcal{F}})\times \mathbb{C}$ 
 (or $\mathbb{D}$) does not admit a Levi-flat $C^\infty$-embedding into any complex 3-manifold (Example \ref{eg:reeb_times_C}). 
See \S \ref{section:eg} for further examples. 
As combining Theorem \ref{thm:main} and \cite[Example 1]{BI}, we obtain the corollary below. 
This seems to be significant in Levi-flat studies about 
not only embedding problems but also moduli problems as mentioned at the beginning of this introduction. 

\begin{corollary}\label{cor:main}
There exist Levi-flat 5-manifolds $(M_{i},\mathcal{F}_{i},J_{\mathcal{F}_{i}})$ for $i=1,2$ 
which satisfy the followings: 
(i) $(M_{1},\mathcal{F}_{1},J_{\mathcal{F}_{1}})$ does not admit 
a Levi-flat $C^\infty$-embedding into any complex 3-manifold, 
(ii) $(M_{2},\mathcal{F}_{2},J_{\mathcal{F}_{2}})$ admits a Levi-flat $C^\infty$-embedding into $\mathbb{C}^3$, and 
(iii) $(M_{1},\mathcal{F}_{1})$ and $(M_{2},\mathcal{F}_{2})$ are diffeomorphic as $C^{\infty}$-foliated manifolds. 
\end{corollary}

Theorem \ref{thm:main} can be obtained from the following: 

\begin{theorem}\label{thm:main_higher_dim}
Let $(M, \mathcal{F},J_{\mathcal{F}})$ be a Levi-flat $(2n+1)$-manifold, 
$L$ a leaf of $\mathcal{F}$, 
$C$ a compact hypersurface of $L$ containing 
an elliptic curve $E$ as a submanifold. 
Assume that 
there exists a neighborhood $\mathcal{U}$ of $C$ in $M$ such that 
$(i)$ the holonomy group $\mathcal{H}(\mathcal{U}\cap L)$ is isomorphic to $\mathbb{Z}$, generated by 
a $C^{\infty}$-flat contracting holonomy along an element of $\pi_{1}(E)$, 
$(ii)$ there exists a $C^\infty$-retraction $p\colon \mathcal{U}\to \mathcal{U}\cap L$ 
whose restriction on each leaf is locally biholomorphic and surjective, 
and $(iii)$ there exists 
a holomorphic function $f$ defined on $\mathcal{U}\cap L$ such that ${\rm div}(f)=C$.
Then $(M, \mathcal{F},J_{\mathcal{F}})$ does not admit 
a Levi-flat $C^\infty$-embedding into any complex $(n+1)$-manifold. 
\end{theorem}


The outline of the proof of 
Theorem \ref{thm:main_higher_dim} is based on that of \cite[Theorem 1]{B}. 
Assume that there exists a Levi-flat $C^\infty$-embedding of $(M, \mathcal{F})$ 
into a complex manifold $X$. 
We will show the existence of a suitable holomorphic function $h$ 
defined on an open neighborhood $W$ of $C$ in $X$ 
and lead to the contradiction by applying the maximum principle to the harmonic function ${\rm Re}\,h$. 
In the proof of \cite[Theorem 1]{B}, the existence of $h$ 
is shown by 
Ueda's theory on a neighborhood structure of 
a complex hypersurface $S$ in $X$ with the $U(1)$-flat normal bundle (i.e. $N_{S/X}\in H^1(S, U(1))$). 
Ueda classified the pair $(S, X)$ by using so-called {\it Ueda type} ``${\rm type}(S, X)$" and gave a sufficient condition for the line bundle $\mathcal{O}_X(S)$ to be $U(1)$-flat on an open neighborhood of $S$ (\cite[Theorem 3]{U}, see also Theorem \ref{thm:ueda_3}). 
However, in our case, one cannot apply Ueda's theorem because $L$ need not to be compact. 
In order to show the existence of $h$, we will use a codimension-2 analogue of Ueda's theory. 
Let $X$ be a complex manifold, $S$ a complex hypersurface of $X$, 
and $C$ a compact complex hypersurface of $S$ 
such that $N_{S/X}$ is $U(1)$-flat on an open neighborhood of $C$ in $S$. 
In \cite{K}, the first author posed the codimension-2 analogue of 
the notion of Ueda type (``${\rm type}(C, S, X)$", \cite[\S 3.1]{K}, 
see also \S \ref{section:review_ueda_theory})
and gave a sufficient condition for the line bundle $\mathcal{O}_X(S)$ 
to be $U(1)$-flat on a neighborhood of $C$ (\cite[Theorem 1]{K}). 
However, recently we found some mistakes in the proof of \cite[Theorem 1]{K}. 
In the present paper, we pose the notion ``the extension type", which also reflects a neighborhood structure of $C$ (see Definition \ref{def:infinexttype}), and show the following theorem as a corrected version of \cite[Theorem 1]{K}: 

\begin{theorem}\label{thm:ueda_nbhd}
Let $X$ be a complex manifold, $S$ a complex non-singular hypersurface of $X$, 
and $C$ a compact complex non-singular hypersurface of $S$. 
Assume that (i) $N_{S/X}$ is $U(1)$-flat and torsion on an open neighborhood $V$ of $C$ in $S$,  
(ii) there exists a holomorphic function $f$ 
defined on $V$ such that ${\rm div}(f)=C$, and  
(iii) ${\rm type}(C, S, X)=\infty$ and $f$ is of extension type infinity. 
Then  $\mathcal{O}_X(S)$ is $U(1)$-flat on an open neighborhood $W$ of $C$ in $X$. 
Moreover, 
there exists a complex hypersurface $Y$ of $W$ which intersects $S$ transversally along $C$. 
\end{theorem}


The organization of the paper is as follows. 
In \S 2, we review the basics of Levi-flat manifolds (\S 2.1), 
Ueda's theory (\S 2.2) and their correspondence (\S 2.3).
In \S 3, we will discuss the well-defindness of types (\S 3.2) 
and show Theorem \ref{thm:ueda_nbhd} (\S 3.4), which is a key proposition in this paper.
In \S 3.1, 
we pose two new concepts: \textit{the extension class} and \textit{the jet extension property}.  
The former one is an obstruction class for the extension of defining functions on hypersurfaces. 
The latter one is a sufficient condition for the well-defindness of types.
In \S 3.3, under suitable Levi-flat situations, we show the jet extension property. 
In \S 4, we prove main results: Theorem \ref{thm:main} and Theorem \ref{thm:main_higher_dim} by using Theorem \ref{thm:ueda_nbhd}. 
Finally, in \S 5, we give some examples and show Corollary \ref{cor:main}.
 
Throughout this paper, we assume that all manifolds and foliations are non-singular and $C^{\infty}$-smooth.

\section{Preliminaries}
We review Levi-flat manifolds and Ueda's theory. 
For basics of foliation theory, we refer \cite{CC}, \cite{HM}. 
For Ueda's theory, see \cite{U}, \cite{N} and \cite{K}. 


\subsection{Levi-flat manifolds}\label{section:terms}
Let $M$ be a smooth manifold of dimension $2n+1$ 
with a smooth foliation $\mathcal{F}$ of real codimension $1$. 
A foliated manifold $(M,\mathcal{F})$ is said 
to be equipped with a {\it leafwise complex structure} $J_{\mathcal{F}}$ if 
there exists a system of foliation charts $\{(\mathcal{U}_j,\varphi_j)\}$ 
which consists of an open covering $\{\mathcal{U}_j\}$ of $M$ and homeomorphisms 
$\varphi_{j}\colon \mathcal{U}_{j} \to \varphi_{j}(\mathcal{U}_{j}) = \Omega \times (0,1) 
\subset \mathbb{C}^{n}\times \mathbb{R}=\{(z_{j},u_{j})\}$, 
where $\Omega$ is an open set in $\mathbb{C}^{n}$. 
The transition maps form
\[
\varphi_{j}\circ \varphi^{-1}_{k}(z_{k},u_{k})=(f_{jk}(z_{k},u_{k}),g_{jk}(u_{k}))   
\]
where $f_{jk}$'s are holomorphic in $z_{k}$ and $f_{jk}$'s and $g_{jk}$'s are smooth in $u_{k}$. 
The triple $(M,\mathcal{F},J_{\mathcal{F}})$ is called a {\it Levi-flat manifold}.
We say $\varphi_{j}^{-1}(\Omega \times \{u_{j}^0\})$ \textit{a plaque} 
and $\varphi_{j}^{-1}(\{z_{j}^0\} \times (0,1))$ \textit {a transversal} of $\mathcal{F}$. 
An equivalence class of plaques is called \textit{a leaf} of $\mathcal{F}$, and 
in this situation, each leaf is a complex manifold. 
Hence $M$ is foliated by complex manifolds 
whose complex structures vary smoothly on the transverse direction. 
Basic examples come from real codimension 1 invariant sets of holomorphic foliations. 
A real hypersurface $M$ in a complex manifold $X$ is 
said to be \textit{Levi-flat} if $M$ is foliated by complex hypersufaces in $X$. 
An embedding from a Levi-flat manifold $(M,\mathcal{F},J_{\mathcal{F}})$ into 
a complex manifold $X$ 
is called a {\it Levi-flat embedding} if its image is a Levi-flat real hypersurface in $X$. 

Take a closed curve $\gamma$ in a leaf $L$ with a base point $p$ 
and a transversal $\mathcal{T}_{p}$ through $p$. 
As moving $\mathcal{T}_{p}$ in foliation charts around $\gamma$, 
the return map determines a local diffeomorphism $\varphi$ on $\mathcal{T}_{p}$. 
Its germ $\varphi_{0}$ is called the \textit{holonomy} of $\mathcal{F}$ along $\gamma$, 
which depends only on the homotopy class of $\gamma$.
The holonomy $\varphi_{0}$ is said to be \textit{contracting} if $|\varphi(u)|<|u|$ near $0$, 
$u\not= 0$, where $u$ is a coordinate on $\mathcal{T}_{p}$.
The holonomy $\varphi_{0}$ is said to be \textit{$C^{r}$-flat} $(1\leq r \leq \infty)$ if 
$\varphi$ is tangent to the identity of order $r$ at $0$. 
Let $G=\textrm{Diff}^{\infty}(\mathbb{R};0)^{\delta}$ be 
the group of germs at $0$ of local $C^{\infty}$ diffeomorphisms 
of $\mathbb{R}$ fixing $0$ with the discrete topology. 
The above correspondence gives 
a homomorphism 
\[
\rho\colon\pi_{1}(L,p)\to G,
\]
called the \textit{holonomy homomorphism} of $\mathcal{F}$ along $L$. 
The image of $\rho$ is called the holonomy group of $\mathcal{F}$ along $L$, denoted by $\mathcal{H}(L)$.

\begin{example}[The Hopf construction of a Reeb component]\label{eg:reeb}
We introduce a construction of \textit{a Reeb component} which will be a basic example 
to which the main theorem apply. 
Set  
\[
\tilde{R}=(\mathbb{C}^{n}\times [0,\infty)) \backslash \{(0',0)\}
\]
with coordinates $(z, u)$ and $0'$ is the origin in $\mathbb{C}^n$.

Choose a constant $\lambda \in \mathbb{C}$ with $|\lambda |>1$ and 
a diffeomorphism $\varphi \in \textrm{Diff}^{\infty}([0,\infty))$ satisfies $\varphi(u)=u+f(u)$
where $f:[0,\infty)\to [0,\infty)$ with $f(0)=0$ is a strictly increasing smooth function and 
$f^{(k)}(0)=0$ for any $k\geq 0$. 
Consider the $\mathbb{Z}$-action on $\tilde{R}$ generated by 
$$\gamma (z, u) = (\lambda z, \varphi(u)).$$ 
The quotient manifold $R=\tilde{R}/\langle \gamma \rangle^{\mathbb{Z}}$ 
is a Levi-flat manifold whose foliation is induced by  
$\{\mathbb{C}^n\times \{u\}\}_{u\in [0,\infty)}$. 
The underlying foliated manifold is called \textit{a Reeb component}.
The Reeb component is diffeomorphic to $D^{2n}\times S^1$ and 
its boundary leaf $L$ is a Hopf manifold. 
By the construction, $\mathcal{F}$ has $C^{\infty}$-flat contracting holonomy along $L$.
When $n=1$ and $\lambda=\exp(2\pi)$, 
the boundary leaf $L$ is biholomorphic to the torus $\mathbb{C}/(\mathbb{Z}+\mathbb{Z}\sqrt{-1})$.  
Take a copy of the Reeb component and attach along the boundary leaves by 
a biholomorphism induced by a map $z\mapsto \sqrt{-1}z$. 
Then we have a Levi-flat manifold $(S^3,\mathcal{F}_{\rm Reeb}, J_{\mathcal{F}})$, 
whose foliation is called \textit{a Reeb foliation} on $S^3$.  
\end{example}


\subsection{Review of Ueda theory}\label{section:review_ueda_theory}

In this subsection, let us recall briefly Ueda's theory and its variant. 

\subsubsection{The original Ueda theory}
Ueda's theory \cite{U} is a neighborhood structure theory of a complex hypersurface 
with the topologically trivial normal bundle. 
Let $S$ be a compact complex hypersurface of a complex manifold $X$ 
with the {\it $U(1)$-flat} normal bundle: i.e. $N_{S/X}\in H^1(S, U(1))$. 
Then the line bundle $\mathcal{O}_X(S)$ defined by the divisor $S$ is topologically trivial on a tubular neighborhood of $S$ in $X$. 
For such a pair $(S, X)$, Ueda defined an obstruction class $u_n(S, X)\in H^1(S, N_{S/X}^{-n})$ for each $n\geq 1$ (see below for the definition). 
By using these obstruction classes, he gave a sufficient condition for $\mathcal{O}_X(S)$ to be $U(1)$-flat around $S$. 

\begin{theorem}[a part of {\cite[Theorem 3]{U}}]\label{thm:ueda_3}
Let $X$ be a complex manifold and $S$ a smooth compact hypersurface of $X$. 
Assume that $N_{S/X}$ is a torsion element of the Picard group and that the pair $(S, X)$ is of infinite type (i.e. $u_n(S, X)=0$ for each integer $n\geq 1$). 
Then there exists a neighborhood $V$ of $S$ in $X$ such that the line bundle $\mathcal{O}_V(S)$ is $U(1)$-flat. 
In particular, there exists a holomorphic function $f$ on a neighborhood of $S$ such that $f$ vanishes only along $S$ with multiplicity $m$, where $m$ is the order of $N_{S/X}$. 
\end{theorem}

\begin{remark}
In \cite[Theorem 3]{U}, Ueda gave a sufficient condition for the $U(1)$-flatness of the line bundle $\mathcal{O}_X(S)$ around $S$ not only when $N_{S/X}$ is torsion, but also when $N_{S/X}$ enjoys the ``Siegel type condition" ($N_{S/X}\in\mathcal{E}_1(S)$, see Remark \ref{rmk:case_E_1}). 
The $U(1)$-flatness of the line bundle $\mathcal{O}_X(S)$ around $S$ can be interpreted as the existence of a holomorphic foliation on a neighborhood of $S$ which has $S$ as a leaf with a $U(1)$-linear holonomy. 
In this sense, Ueda theory is related to 
the local existence problem of the holomorphic foliations (see also \cite{CLPT}).
\end{remark}

Let us explain the definition of the obstruction classes $u_{n}(S, X)$ along \S 2 in \cite{U}. 
Take a sufficiently fine open covering $\{V_j\}$ of $S$. 
From the assumption, 
$N_{S/X}$ is represented by $\{(V_{jk}, t_{jk})\}$ for some constants $t_{jk}\in U(1)$, 
where $V_{jk}:=V_j\cap V_k$ (see \cite[\S 1.1]{U}). 
Let $W$ be a tubular neighborhood of $S$ in $X$ 
and $\{W_j\}$ be sufficiently fine open covering of $W$. 
Assume that $V_j=W_j\cap S$ for each $j$, and $V_{jk}=\emptyset$ iff $W_{jk}=\emptyset$. 
Take a coordinates system $(z_j, w_j)$ on $W_j$ such that $z_j$ is a coordinate on $V_j$, 
$w_j$ is a defining function of $V_j$ and $(w_j/w_k)|_{V_{jk}}\equiv t_{jk}$ on $V_{jk}$ for each $j,k$. 
Let $n$ be a positive integer. 
We say that a system $\{(W_j, w_j)\}$ is {\it of type $n$} if ${\rm mult}_{V_{jk}}(t_{jk}w_k-w_j)\geq n+1$ holds (i.e. the $n$-jet of $t_{jk}w_k-w_j$ vanishes along $V_{jk}$) on each $W_{jk}$. If there exists a system $\{(W_j, w_j)\}$ of type $n$, 
the Taylor expansion of $t_{jk}w_k$ in the variable $w_j$ on $W_{jk}$ around $\{w_j=0\}$ can be written in the form 
\[
t_{jk}w_k=w_j+f_{jk}^{(n+1)}(z_j)\cdot w_j^{n+1}+O(w_j^{n+2})
\]
for some holomorphic function $f^{(n+1)}_{jk}$ defined on $V_{jk}$. 
It is known that $\{(V_{jk}, f^{(n+1)}_{jk})\}$ satisfies the cocycle condition (see \cite[p. 588]{U}). 

\begin{definition}
Suppose that there exists a system of type $n$. 
Then the cohomology class 
\[
u_n(S, X):=[\{(V_{jk}, f^{(n+1)}_{jk})\}]\in H^1(S, N_{S/X}^{-n})
\]
is called {\it the $n$-th Ueda class} of the pair $(S, X)$. 
\end{definition}

Unfortunately, the $n$-th Ueda class may depend on the choice of a system of type $n$. 
Thus, strictly speaking, this class should be written as ``$u_n(S, X; \{w_j\})$" in the general setting. 
However, under some configurations, the condition ``$u_n(S, X)=0$" does not depend on the choice of a system of type $n$. 
For example, it follows from Ueda's argument in \cite[p. 588]{U} that the $n$-th Ueda class is well-defined up to non-zero constant multiples if $S$ is compact, or if $N_{S/X}$ is holomorphically trivial and $H^0(S, \mathcal{O}_S)=\mathbb{C}$ holds. 
We will also show the well-definedness of the condition ``$u_n(S, X)=0$'' for a slightly generalized configuration which appears in the proof of our main results (under ``the jet extension property", see \S \ref{section:well-def}). 

At first, we consider the case where the condition ``$u_n(S, X)=0$'' does not depend on the choice of a system of type $n$. 
In this case, it is known that $u_n(S, X)=0$ if and only if there exists a system of type $n+1$. 
Thus, in this case, one 
of the following occurs. 
\begin{itemize}
\item There exists an integer $n\in\mathbb{Z}_{>0}$ such that $u_m(S, X)$ can be defined only when $m\leq n$, $u_m(S, X)=0$ for $m<n$, and $u_n(S, X)\not=0$. 
\item For every integer $n\in\mathbb{Z}_{>0}$, $u_n(S, X)$ can be defined and $u_n(S, X)=0$. 
\end{itemize}

We denote by ``${\rm type}\,(S, X)$'' the supremum of the set of integers $n$ such that $u_m(S, X)$ can be defined and $u_m(S, X)=0$ for each $m<n$. 
In this paper, we will treat the notion ``of infinite type'' in a slightly general setting (even if the condition ``$u_n(S, X)=0$'' depends on the choice of a system of type $n$) by adopting the following definition: 

\begin{definition}
We say that the pair $(S, X)$ is {\it of infinite type} if $u_n(S, X; \{w_j\})=0$ holds for each $n\geq 1$ and for {\it any} system $\{(W_j, w_j)\}$ of type $n$. 
\end{definition}

\subsubsection{The codimension-$2$ Ueda theory}\label{section:codim2ueda_review}
In \cite{K}, a codimension-$2$ analogue of this Ueda's theorem is posed. 
Let $X$ be a complex manifold, $S$ a complex hypersurface of $X$, 
and $C$ a compact complex hypersurface of $S$ 
with the normal bundles $(N_{S/X},\{t_{jk}\})$ and $(N_{C/S},\{s_{jk}\})$. 
Assume that $N_{S/X}$ is $U(1)$-flat on an open neighborhood $V$ of $C$ in $S$ 
and $N_{C/S}$ is also $U(1)$-flat: $t_{jk}, s_{jk}\in U(1)$. 
Let $W$ be a tubular neighborhood of $C$ in $X$ such that $W\cap S=V$. 
Take a sufficiently fine open covering 
$\{U_j\}$ of $C$, $\{V_j\}$ of $V$, and $\{W_j\}$ of $W$ 
such that $V_j=W_j\cap S$, $U_j=V_j\cap C$, and $V_{jk}=\emptyset$ iff $W_{jk}=\emptyset$. 
Extend a coordinates system $x_j$ of $U_j$ to $W_j$. 
Let $y_{j}$ be a defining function of $U_{j}$ in $V_{j}$ 
and $w_j$ a defining function of $V_j$ in $W_{j}$. 
Denote an extension of $y_{j}$ on $W_{j}$ by $z_{j}$, that is $z_{j}|_{V_{j}}=y_{j}$,
so that $(x_{j},z_{j},w_{j})$ is a coordinates system on $W_{j}$ (Figure 1). 
From the same argument as in \cite[\S 2]{U}, 
we may assume that $({w_j}/{w_k})|_{V_{jk}}\equiv t_{jk}$ and $({y_j}/{y_k})|_{U_{jk}}\equiv s_{jk}$. 

\begin{figure}[htbp]
 \begin{center}
  \includegraphics[width=100mm]{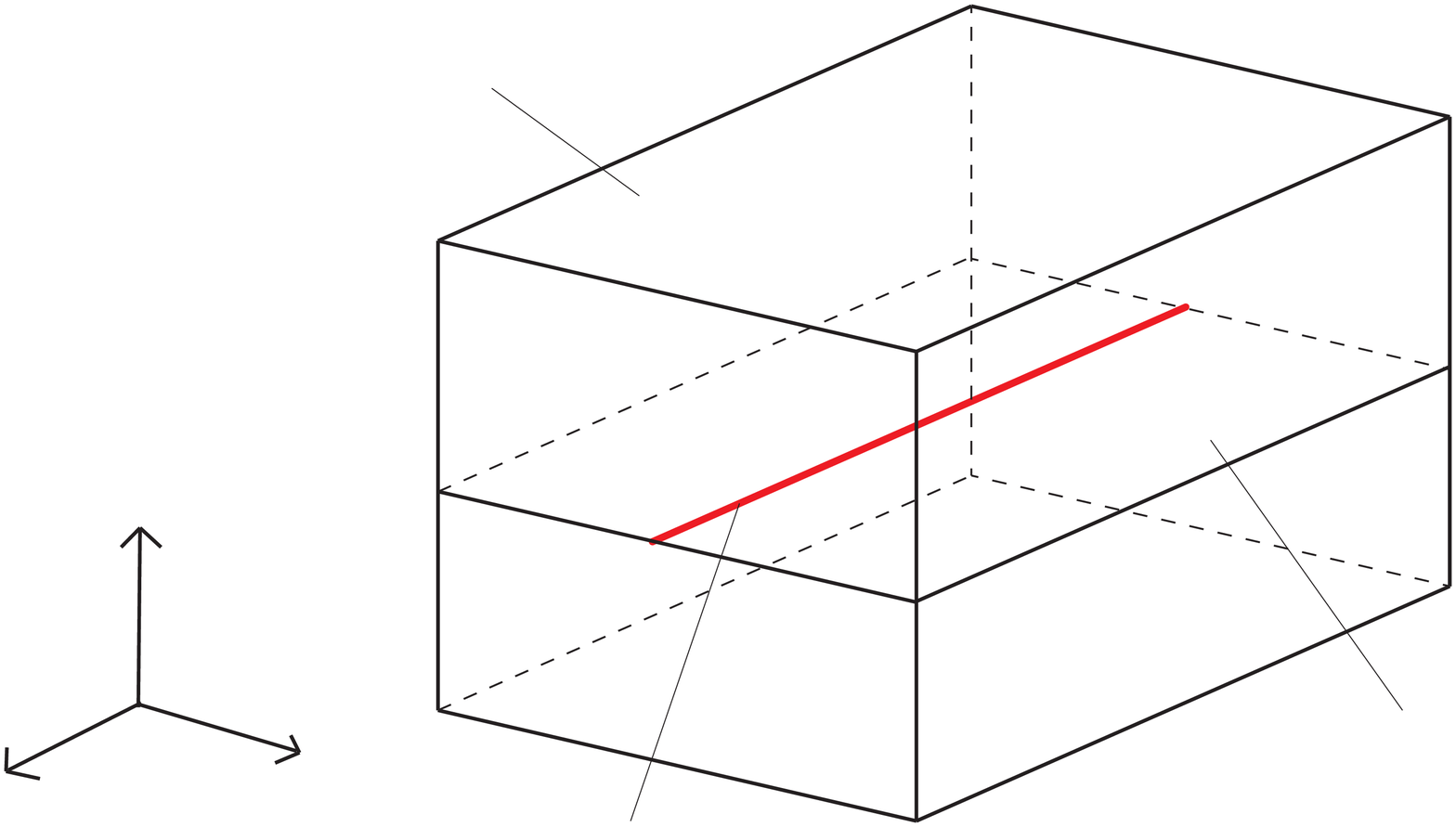}
 \end{center}
 \begin{picture}(0,0)
\put(-25,3){$U_{j}\subset C$}
\end{picture}
 \begin{picture}(0,0)
\put(130,25){$V_{j}\subset S$}
\end{picture}
 \begin{picture}(0,0)
\put(-70,170){$W_{j}\subset X$}
\end{picture}
 \begin{picture}(0,0)
\put(-150,15){$x_{j}$}
\end{picture}
 \begin{picture}(0,0)
\put(-95,18){$z_{j}$}
\end{picture}
 \begin{picture}(0,0)
\put(110,5){$(y_{j}:=z_{j}|_{V_{j}})$}
\end{picture}
 \begin{picture}(0,0)
\put(-140,80){$w_{j}$}
\end{picture}
 \caption{}
 \label{fig:one}
\end{figure}

Assume that our system $\{(W_j, w_j)\}$ is of type $(n, m)$ in the sense of \cite[Definition 4]{K}. 
i.e. the coefficient of $w_j^\nu z_j^\mu$ in the Taylor expansion of the function $t_{jk}w_k-w_j$ 
in the variables $w_j$ and $z_j$ around $U_{jk}$ is equal to zero if $(\nu, \mu)< (n+1, m)$ holds in the lexicographical order, namely 
$(a, b)\leq (a', b')\overset{\mathrm{def}}{\Leftrightarrow} a< a'\ \text{or}\ ``a=a'\ \text{and}\ b\leq b'"$. 
Then 
the function $t_{jk}w_k$ can be expanded in the variable $w_j$ as 
\begin{equation}\label{eq:def_u_n}
 t_{jk}w_k=w_j+f_{jk}^{(n+1)}(x_j, z_j)\cdot w_j^{n+1}+O(w_j^{n+2})
\end{equation}
and the expansion of $f_{jk}^{(n+1)}(x_j, z_j)$ in the variable $z_j$ is written by 
\begin{equation}\label{eq:def_u_nm}
 f_{jk}^{(n+1)}(x_j, z_j)=g_{jk}^{(n+1, m)}(x_j)\cdot z_j^m+O(z_j^{m+1}). 
\end{equation}
It can be shown that a system $\{(U_{jk}, g_{jk}^{(n+1, m)})\}$ enjoys the cocycle condition for the line bundle $N_{S/X}|_C^{-n}\otimes N_{C/X}^{-m}$ (\cite[Proposition 2]{K}). 
The $1$-cocycle $\{(U_{jk}, g_{jk}^{(n+1, m)})\}$ defines 
an element of $H^1(C, N_{S/X}|_C^{-n}\otimes N_{C/S}^{-m})$, 
denoted by $u_{n, m}(C, S, X)$ and called {\it the $(n, m)$-th Ueda class} of the triple $(C, S, X)$. 

In \S \ref{pfthm:ueda_nbhd},  
to show Theorem \ref{thm:ueda_nbhd}, 
we will refine the systems $\{w_{j}\}$ and $\{z_{j}\}$ \textrm{keeping fixed} $\{y_{j}\}$ and $\{x_{j}\}$.  
As we show in Lemma \ref{lemma_well-def_of_u_n_m_z}, 
the $(n, m)$-th Ueda class of $(C,S,X)$ depends only on the the choice of a system $\{w_{j}\}$ of type $(n, m)$. 
So the class should be written as  ``$u_{n, m}(C, S, X; \{w_j\})$". 
However, under some configurations, the condition  ``$u_{n, m}=0$" {\it does not} depend on 
such a system as in the case of $u_n$ (see \S \ref{section:well-def}).

\begin{definition}\label{def:infintype}
We say that the triple $(C, S, X)$ is {\it of infinite type} if $u_{n, m}(C, S, X; \{w_j\})=0$ holds for each $n\geq 1, m\geq 0$ and for {\it any} system $\{(W_j, w_j)\}$ of type $(n, m)$. 
\end{definition}

\begin{remark}\label{rmk:un_unm}
If the pair $(S, X)$ is of infinite type and the condition \linebreak``$u_{n, m}(C, S, X; \{w_j\})=0$" does not depend on the choice of the system $\{w_j\}$ of type $(n, m)$, then the triple $(C, S, X)$ is of infinite type.
\end{remark}


\subsection{Barrett-Fornaess' theorem and Appendix in \cite{B}}\label{section:Barrett-Fornaess}

Let $(M,\mathcal{F})$ be a Levi-flat hypersuface in a complex manifold $X$. 
In \cite{BF}, Barrett and Fornaess showed that 
if $M$ is of class $C^r$ ($r\geq 2$) then so is $\mathcal{F}$. 
To show this, they constructed local holomorphic coordinates 
approximating a foliation chart of $(M,\mathcal{F})$. 
In the appendix of \cite{B}, by using this, 
Barrett clarified a relation between 
the Ueda type of a leaf and the $C^r$-flatness of holonomy. 
This point of view appears  
in \S \ref{section:jetextproperty} and \S \ref{proof of theorem}. 

More precisely, from the condition of $C^{r}$-flatness, we obtain the following data 
(Appendix in \cite{B}): 
A system $\{(W_{j},w_{j})\}$ of holomorphic local defining functions for $L$ in $X$ and 
a system $\{(\mathcal{U}_{j},u_{j})\}$ of leafwise constant $C^{\infty}$ defining functions for $L$ in $M$ 
which satisfy $\mathcal{U}_{j}=W_{j}\cap M$ and 
the following conditions for each $j, k$ $\colon$$\vspace{2mm}$\\
(i) $w_{k}-w_{j} = O(w_{j}^{r+1})$ on $W_{jk}$,$\vspace{1mm}$\\ 
(ii) $u_{k}-u_{j}=O(u_{j}^{r+1})$ on $\mathcal{U}_{jk}$, $\vspace{1mm}$\\
(iii) (Im $w_{j})|_{\mathcal{U}_{j}}$ $= o(|w_{j}|^{r})$ on $\mathcal{U}_{j}$ and$\vspace{1mm}$\\
(iv) (Re $w_{j} )|_{\mathcal{U}_{j}}$ $=$ $u_{j}+o(u_{j}^r)$ on $\mathcal{U}_{j}$. $\vspace{2mm}$\\
The condition (ii) is the $C^{r}$-flatness of $\mathcal{F}$. 
The condition (iii) was shown in \cite{BF}. 
To get such a system, 
for a holomorphic local defining function of $L$, 
we inductively correct it so as to vanish the jet of its imaginary part along $M$.  
In the arguments in \S \ref{section:well-def} and \S \ref{section:jetextproperty}, 
we will follow a similar procedure.  
Note that, by the condition (i), 
the $C^1$-flatness implies that the normal bundle $N_{L/X}$ is holomorphically trivial, that is, $t_{jk}\equiv 1$.

\begin{proposition}[Barrett, Appendix in {\cite{B}}]\label{prop:BAppend}
Let $(M,\mathcal{F})$ be a Levi-flat hypersurface in a complex manifold $X$ 
and $L$ a (possibly non-compact) leaf of $\mathcal{F}$ embedded in $X$. 
If $\mathcal{F}$ has $C^{r}$-flat holonomy along $L$, 
then 
there exists a system $\{(W_{j},w_{j})\}$ 
of defining holomorphic functions for each plaque of $L$ 
which satisfy the following conditions: 
\[
w_{k}-w_{j} = O(w_{j}^{r+1})~\textrm{on}~ W_{jk}\quad\textrm{and}\quad
({\rm Im}~w_{j})|_{\mathcal{U}_{j}} = o(|w_{j}|^{r})~\textrm{on}~\mathcal{U}_{j}. 
\]
In particular, the normal bundle $N_{L/X}$ is holomorphically trivial. 
If $L$ is compact, then the pair $(L,X)$ is of infinite type in the sense of Ueda. 
\end{proposition}

\section{Codimension 2 Ueda theory and jet extension property}\label{section:open_analogue_ueda_theory}
Let the triple $(C,S,X)$ and its normal bundles be as in \S \ref{section:codim2ueda_review}. 
Take tubular neighborhoods $V$ and $W$, open coverings $\{U_j\},\{V_j\}$ and $\{W_j\}$, 
and functions $\{x_j\},\{y_j\},\{z_j\}$ and $\{w_j\}$ as in \S \ref{section:codim2ueda_review}. 
In addition, assume that $\{y_j\}$ satisfies $s_{jk}y_{k}=y_{j}$ on $V_{jk}$. 
In the following sections, we will refine the system $\{w_{j}\}$ and 
the extension $\{z_{j}\}$ of $\{y_{j}\}$ 
\textit{keeping fixed $\{x_{j}\}$ and $\{y_{j}\}$}. 
In what follows, we consider only a tubular neighborhood of the compact set $C$, that is,
regard $S=V$ and $X=W$. 

\subsection{The extension type}\label{section:ext_type}
We introduce the notion of the \textit{type} of extension $\{z_{j}\}$ of $\{y_{j}\}$.

\begin{definition}
Let $\{y_{j}\}$ and $\{z_{j}\}$ be as above. 
For $n\geq1, m\geq 0$,
the extension $\{z_{j}\}$ of $\{y_{j}\}$ is \textit{of type $(n,m)$} if 
the expansion in $w_j$ forms 
\[
s_{jk}z_k=z_j+p_{jk}^{(n)}(x_j, z_j)\cdot w_j^{n}+O(w_j^{n+1})
\]
and the expansion of $p_{jk}^{(n)}(x_j, z_j)$ in $z_j$ is written by
\[
p_{jk}^{(n)}(x_j, z_j)=q_{jk}^{(n, m)}(x_j)\cdot z_j^m+O(z_{j}^{m+1})
\]
for each $j$.
\end{definition}
The $1$-cochain $\{(U_{jk}, q_{jk}^{(n, m)})\}$ satisfies the cocycle condition 
valued on $N_{S/X}|^{-n}_{C}\otimes N_{C/S}^{-m+1}$. 
This $1$-cocycle defines a cohomology class in $H^{1}(C, N_{S/X}|^{-n}_{C}\otimes N_{C/S}^{-m+1})$, 
which is denoted by $v_{n, m}(C, S, X; \{z_j\})$. 

\begin{definition}\label{def:infinexttype}
We say that $\{y_{j}\}$ is \textit{of extension type infinity} if 
$v_{n, m}(C, S, X;\{z_j\})=0$ holds for each $n\geq 1, m\geq 0$ and 
for {\it any} type $(n, m)$ extension $\{z_{j}\}$ of $\{y_{j}\}$. 
\end{definition}

\subsection{Well-definedness of the types and jet extension property}\label{section:well-def}

As we mentioned, 
the obstruction classes $u_{n}$, $u_{n,m}$ and also $v_{n,m}$ may depend on the 
choice of systems. 
Let us give sufficient conditions for the well-definedness of 
$``u_{n}=0"$, $``u_{n, m}=0"$ and $``v_{n, m}=0"$ 
and proof them in the next three subsections. 
These proofs are in essentially the same manner as Lemma \ref{lem:jet_ext_welldef_unm} below.
In the following subsections, we often assume that $t_{jk}\equiv 1$. 
Note that, this will be automatically satisfied in our Levi-flat situation. See Proposition \ref{prop:BAppend}. 

Ueda showed the well-definedness of Ueda type for $(S,X)$ 
when $S$ is compact without boundary (\cite[p. 588]{U}). 
The key point of the proof is that the restriction map 
$H^0(X, \mathcal{O}_X)\to H^0(S, \mathcal{O}_S)=\mathbb{C}$ is surjective. 
Let us refine this argument (see also the proof of \cite[Lemma 1.2]{N}). 
Here we introduce a key concept in the present work. 

\begin{definition} Let $n$ be a positive interger.
We say 
the pair $(S,X)$ satisfies \textit{the $n$-jet extension property} if 
the restriction map 
$H^0(S, \mathcal{O}_X/I_S^{n+1})\to H^0(X, \mathcal{O}_X/I_S) =H^0(S, \mathcal{O}_S)$ is surjective, 
where $I_S\subset \mathcal{O}_X$ is the defining ideal sheaf of $S$. 
Also say that the pair $(S,X)$ satisfies \textit{the jet extension property} if 
it does the $n$-th jet extension property for any $n$. 
\end{definition}

\subsubsection{Well-definedness of the condition $``u_{n}=0"$}

\begin{lemma}\label{lem:jet_ext_welldef_unm}
The condition $u_n(S, X; \{w_j\})=0$ does not depend on the choice of the system $\{w_j\}$ of type $n$ if
the pair $(S,X)$ satisfies the $n$-jet extension property and $t_{jk}\equiv 1$.
\end{lemma}

\begin{proof}
Let $\{(W_j, w_j)\}$ and $\{(W_j, \widehat{w}_j)\}$ be systems of type $n$. 
Set $V_{j}=W_{j}\cap S$ for each $j$. 
Consider the expansion 
\[
\widehat{w}_j=a^{(1)}_j\cdot w_j+O(w_{j}^2)
\]
by $w_j$. 
Let us refine the system $\{\widehat{w}_{j}\}$ keeping the Ueda class up to nonzero multiplications.
First note that, each $a^{(1)}_j$ 
defines a global function $a^{(1)}\in H^0(S, \mathcal{O}_S^*)$.
In fact, 
\[
a^{(1)}_j|_{V_{jk}}=\left.\frac{\widehat{w}_j}{w_j}\right|_{V_{jk}}=\left.\frac{t_{jk}\widehat{w}_k+O(\widehat{w}_j^2)}{t_{jk}w_k+O(w_j^2)}\right|_{V_{jk}}=\left.\frac{t_{jk}\widehat{w}_k}{t_{jk}w_k}\right|_{V_{jk}}=a^{(1)}_k|_{V_{jk}}.
\]
From the $n$-jet extension property, 
we can take nowhere vanishing holomorphic functions $\widetilde{a}^{(1)}_j\colon W_j\to \mathbb{C}$ 
with $\widetilde{a}^{(1)}_j|_{V_{j}}=a^{(1)}_j$ and 
\[
\widetilde{a}^{(1)}_k-\widetilde{a}^{(1)}_j=O(w_j^{n+1})
\]
on $W_{jk}$ for each $j,k$ 
(by shrinking $W_j$ if necessary). 

Set 
\[
\widehat{w}_j':=\widehat{w}_j/\widetilde{a}^{(1)}_j=w_{j}+O(w_{j}^2)
\]
and consider the new system $\{(W_j, \widehat{w}_j')\}$. 
The change of the coordinates is obtained by$\vspace{4mm}$
\\
\centerline{
$
\begin{array}{lll}
\widehat{w}_k'-\widehat{w}_j'&=&
\frac{1}{\widetilde{a}^{(1)}_j}\cdot 
\left(\frac{\widetilde{a}^{(1)}_j}{\widetilde{a}^{(1)}_k}\cdot\widehat{w}_k-\widehat{w}_j\right)\vspace{2mm}\\
&=&\frac{1}{\widetilde{a}^{(1)}_j}\cdot \left(\widehat{w}_k-\widehat{w}_j+O(\widehat{w}_j^{n+2})\right)\vspace{2mm}\\
&=&\frac{1}{\widetilde{a}^{(1)}_j(z_{j},\widehat{w_{j}})}\cdot 
\Big{(}\widehat{f}^{(n+1)}_{jk}(z_{j})\cdot \widehat{w}_{j}^{n+1}
+O(\widehat{w}_j^{n+2})\Big{)}\vspace{3mm}
\\
&=&\Big{(}\widetilde{a}^{(1)}_j(z_{j},\widehat{w_{j}})\Big{)}^n\cdot 
\Big{(}\widehat{f}^{(n+1)}_{jk}(z_{j})\cdot (\widehat{w}'_{j})^{n+1}
+O((\widehat{w}'_j)^{n+2})\Big{)}\vspace{3mm}
\\
&=&\Big{(}a^{(1)}_j(z_{j})\Big{)}^n\cdot 
\widehat{f}^{(n+1)}_{jk}(z_{j})\cdot (\widehat{w}'_{j})^{n+1}
+O((\widehat{w}'_j)^{n+2}).\vspace{3mm}
\end{array}
$}
\\
Here $\widehat{f}_{jk}^{(n+1)}$ is as in the expansion (\ref{eq:def_u_n}) in \S\ref{section:review_ueda_theory}. 
Note that 
$O(w_j)=O(\widehat{w}_j)=O(\widehat{w}'_j)$, 
$\widetilde{a}^{(1)}_j/~\widetilde{a}^{(1)}_k=1+O(\widehat{w}_j^{n+1})$ and 
$a^{(1)}_j(z_{j})$ is the top term of the expansion of $\widetilde{a}^{(1)}_j(z_{j},\widehat{w}_{j})$
in the variable $\widehat{w}_{j}$.
It follows that 
$$u_n(S, X; \{\widehat{w}_j'\})=(a^{(1)})^n\cdot u_n(S, X; \{\widehat{w}_j\}),$$ 
where $a^{(1)}\neq 0$. 
The system $\{\widehat{w}_j\}$ can be replaced by $\{\widehat{w}_j'\}$ for our purpose (For simplicity, we use the notation $\{\widehat{w}_j\}$ again). 
Thus the expansion is written as
\[
\widehat{w}_j=w_j+a^{(2)}_j\cdot w_j^2+O(w_{j}^3). 
\]

In what follows, we will show the equality 
$u_n(S, X; \{w_j\})=u_n(S, X; \{\widehat{w}_j\})$ 
for 
\textit{any} system $\{(W_j, \widehat{w}_j)\}$ of type $n$ with 
\[
\widehat{w}_j=w_j+a^{(\nu)}_j\cdot w_j^{\nu}+O(w_{j}^{\nu+1})
\]
for each $j$. 
Here $\nu\geq 2$ is the infimum over $j$ of the lowest degrees of $\{\widehat{w}_j-w_j\}$. 
The proof is by induction on $\nu$. 
First, a direct calculation shows that 
\begin{equation}\label{change}
\widehat{w}_{k}-\widehat{w}_{j}=
f_{jk}^{(n+1)}(z_{j})\cdot w_{j}^{n+1}+O(w_{j}^{n+2})
+\Big{(}a^{(\nu)}_k(z_{k}(z_{j},0))-a^{(\nu)}_j(z_{j})\Big{)}\cdot w_{j}^{\nu}+O(w_{j}^{\nu+1}). 
\end{equation}
Note that $a^{(\nu)}_k(z_{k}(z_{j},0))$ is the top term of the expansion of $a^{(\nu)}_k(z_{k}(z_{j},w_{j}))$ in the variable $w_{j}$. 
When $\nu \geq n+1$, 
the equality $u_n(S, X; \{w_j\})=u_n(S, X; \{\widehat{w}_j\})$ clearly holds. 

When $\nu < n+1$,
assume that the equality holds for any system of type $n$ with such an expansion for $\nu>\nu_0\geq 2$. 
Consider a system $\{(W_j, \widehat{w}_j)\}$ of type $n$ with 
\[
\widehat{w}_j=w_j+a^{(\nu_0)}_j\cdot w_j^{\nu_0}+O(w_{j}^{\nu_{0}+1}). 
\]
It is easy to check that $\{a^{(\nu_0)}_j\}$ 
defines a global function $a^{(\nu_0)}\in H^0(S, \mathcal{O}_S)$. 
Thus, by the assumption, 
we can take holomorphic functions $\widetilde{a}^{(\nu_0)}_j\colon W_j\to \mathbb{C}$ 
with 
$\widetilde{a}^{(\nu_0)}_j|_{V_{j}}=a^{(\nu_0)}_j$ and 
$\widetilde{a}^{(\nu_0)}_j-\widetilde{a}^{(\nu_0)}_k=O(w_j^{\nu_0+1})$ on $W_{jk}$ for each $j,k$.  
Then set 
\[
\widehat{w}_j':=\widehat{w}_j-\widetilde{a}^{(\nu_0)}_j\cdot w_j^{\nu_0}
\]
and consider the new system $\{(W_j, \widehat{w}_j')\}$. 
A simple computation shows that 
\[
u_n(S, X; \{\widehat{w}_j\})=u_n(S, X; \{\widehat{w}'_j\})
{\quad\textrm{and}\quad
\widehat{w}_j'=w_{j}+O(w_{j}^{\nu_0+1}). }
\] 
By the inductive assumption, we obtain $u_n(S, X, \{w_j\})=u_n(S, X, \{\widehat{w}_j\})$. 
The proof is completed. 
\end{proof}

\subsubsection{Well-definedness of the condition $``u_{n, m}=0"$}\label{welldefuv}

\begin{lemma}\label{lemma_well-def_of_u_n_m_z}
The class $u_{n, m}(C,S,X;\{w_j\})$ does not depend on the choice of the type $(1,0)$ extension $\{z_j\}$ of $\{y_{j}\}$.
\end{lemma}

\begin{proof}
This is immediately from the definition of $u_{n, m}$. 
In fact, after restricting (\ref{eq:def_u_nm})
on $S$, 
we have $g_{jk}^{(n+1, m)}(x_j) =\frac{f_{jk}^{(n+1)}(x_j, y_j)}{y_j^m}\Big{|}_{y_{j}=0}$ and 
it shows the independence.
\end{proof}

\begin{lemma}\label{lemma_well-def_of_u_n_m}
The condition $u_{n,m}(C,S,X; \{w_j\})=0$
does not depend on the choice of the system $\{w_j\}$ of type $(n,m)$ if
the pair $(S,X)$ satisfies the $n$-jet extension property and $t_{jk}\equiv 1$.
\end{lemma}
\begin{proof}
Let $\{(W_j, w_j)\}$ and $\{(W_j, \widehat{w}_j)\}$ be systems of type $(n,m)$. 
Denote the expansion coefficients by 
$f_{jk}^{(n+1)}(x_j, z_j), \widehat{f}_{jk}^{(n+1)}(x_j, z_j), g_{jk}^{(n+1,m)}(x_{j})$
and $\widehat{g}_{jk}^{(n+1,m)}(x_{j})$ 
as in the expansions (\ref{eq:def_u_n}) and (\ref{eq:def_u_nm}) in \S \ref{section:review_ueda_theory}. 
The proof is the same manner as Lemma \ref{lem:jet_ext_welldef_unm}.
After rescaling the defining functions, 
we may assume that 
\[
\widehat{w}_j=w_j+a^{(\nu)}_j(x_j, z_j)\cdot w_j^{\nu}+O(w_j^{\nu+1})
\]
while keeping the desired property for the Ueda type. 

Let us show the critical case $\nu=n+1$
(The other cases can be proved similarly. 
If necessarily, we can replace the system while keeping the desired property). 
From the equation (\ref{eq:def_u_nm}) and (\ref{change}), 
\[
\widehat{g}_{jk}^{(n+1, m)}(x_j)\cdot z_j^m=g_{jk}^{(n+1, m)}(x_j)\cdot z_j^m+a^{(n+1)}_k(x_j,z_j)-a^{(n+1)}_j(x_j,z_j)+O(z_j^{m+1}).
\]
Consider the expansion
\[
a_j^{(n+1)}(x_j, y_j)=a_j^{(n+1,\mu)}(x_j)\cdot y_{j}^{\mu}+O(y_{j}^{\mu+1})
\]
on $V_{j}$, 
where $\mu\geq 1$ is the infimum over $j$ of the lowest degrees of $\{a_{j}^{(n+1)}\}$. 
As repeating the same argument as Lemma \ref{lem:jet_ext_welldef_unm} for $\{a_{j}^{(n+1,\mu)}\}$, 
$\{g_{jk}^{(n+1,m)}\}$, and $\{\widehat{g}_{jk}^{(n+1,m)}\}$ on $(C,S)$, 
we can show the assertion. 
In fact, when $\mu\geq m$, 
the equality $u_{n,m}(C, S, X; \{w_j\})=u_{n,m}(C, S, X; \{\widehat{w}_j\})$ clearly holds.
When $\mu<m$, $\{a_{j}^{(n+1,\mu)}\}$ defines 
the non-zero constant $a^{(n+1,\mu)} \in H^0(C, N_{C/S}^{-\mu})$ 
(recall the assumption $s_{jk}y_{k}=y_{j}$ on $V_{jk}$).
Since $C$ is compact and $N_{C/S}^{-\mu}$ is flat, it follows by applying the maximal value principle 
to the plurisubharmonic function $|a^{(n+1,\mu)}|$ that $a^{(n+1,\mu)}$ is a constant 
(note that especially 
$N_{C/S}^{-\mu}$ is holomorphically trivial).
Then we set 
\[
\widehat{w}_j':=\widehat{w}_j-a^{(n+1,\mu)}\cdot z_{j}^{\mu}w_j^{n+1} = 
w_j+O(z_{j}^{\mu+1})w_{j}^{n+1}+O(w_{j}^{n+2}) 
\]
and consider the new system $\{(W_j, \widehat{w}'_j)\}$. 
By the direct calculation of $\widehat{w}'_k-\widehat{w}'_j$, we obtain 
$u_{n,m}(C,S,X; \{\widehat{w}_j\})=u_{n,m}(C,S,X; \{\widehat{w}'_j\})$. 
As repeating this procedure inductively, it reduces to the case $m=\mu$, so that the proof is completed. 
\end{proof}

\subsubsection{Well-definedness of the condition $``v_{n, m}=0"$} 

\begin{lemma}
The condition $v_{n, m}(C,S,X;\{z_j\})=0$ does not depend on the choice of the system $\{w_j\}$ of type $(1,0)$.
\end{lemma}
\begin{proof}
Let 
$\{z_{j}\}$ be an $(n,m)$ type extension of $\{y_{j}\}$ and
$\{(W_j, w_j)\}$, $\{(W_j, \widehat{w}_j)\}$ be systems of type $(1,0)$.  
The expansion of $z_{j}$ in $\widehat{w}_j$ is denoted by 
\[
 s_{jk}z_k=z_j+\widehat{p}_{jk}^{(n)}(x_j, z_j)\cdot \widehat{w}_j^{n}+O(\widehat{w}_{j}^{n+1}). 
\]
The expansion of $\widehat{w}_{j}$ in $w_{j}$ is written by 
\[
\widehat{w}_j=a^{(1)}_j(x_j, z_j)\cdot w_j+O(w_j^2)
\]
and $\{a^{(1)}_{j}\}$ defines a global function $a^{(1)}\colon S\to \mathbb{C}^{\ast}$ 
as in the proof of Lemma \ref{lem:jet_ext_welldef_unm}. 
Moreover, the top terms $\{a^{(1,0)}_j(x_j)\}$ 
of the expansion of $a^{(1)}_{j}(x_j, y_j)$ in $y_{j}$ on $V_{j}$ defines 
a global (constant) function $a^{(1,0)}\colon C\to \mathbb{C}^{\ast}$. 
Then we have
\[
\widehat{p}_{jk}^{(n)}(x_j, y_j)=(a^{(1)})^{-n} \cdot {p}_{jk}^{(n)}(x_j, y_j)
=(a^{(1,0)})^{-n}\cdot q^{(n, m)}_{jk}(x_j)\cdot y_j^m+O(y_j^{m+1}), 
\]
so that we conclude that $\widehat{q}^{(n, m)}_{jk}(x_{j})=(a^{(1,0)})^{-n}\cdot q^{(n, m)}_{jk}(x_j)$ and the claim follows. 
\end{proof}

Let $\{w_j\}$ be a system of type $(n, 0)$ and 
the extension class $v_{n,m}(C,S,X; \{z_j\})$ is defined. 

\begin{lemma}\label{lem:weldef_v_n_m_z_j}
For each $m\geq 0$, $v_{n,m}(C,S,X; \{z_j\})=0$ for any extension $\{z_{j}\}$ of type $(n,m)$ 
if the pair $(S,X)$ satisfies the $n$-jet extension property and $t_{jk}\equiv 1$,  $s_{jk}\equiv 1$.
\end{lemma}

\begin{proof}
By the assumption $s_{jk}\equiv 1$, $\{y_{j}\}$ defines a global function on $S$. 
From the $n$-jet extension property, there exists an extension $\{z_{j}\}$ of $\{y_{j}\}$ with
\begin{equation}\label{expz}
z_k=z_j+O(w_j^{n+1}), 
\end{equation}
so that $v_{n,m}(C,S,X; \{z_{j}\})=0$ holds. 

Let $\{\widehat{z}_{j}\}$ be another extension of $\{y_j\}$ with type  $(n,m)$.  
The expansion coefficients are denoted by 
$p_{jk}^{(n+1)}(x_j, z_j),$ $\widehat{p}_{jk}^{(n+1)}(x_j, \widehat{z}_j),$ 
$q_{jk}^{(n+1,m)}(x_{j})$
and $\widehat{q}_{jk}^{(n+1,m)}(x_{j})$. 
Since $\widehat{z}_{j}|_{V_{j}}=z_j|_{V_j}=y_j$, 
the expansion of $\widehat{z}_{j}$ in $w_{j}$ is written by
\[
\widehat{z}_j= z_j+a^{(\nu)}_j(x_j, z_j)\cdot w_{j}^\nu+O(w_{j}^{\nu+1}), 
\]
for $\nu \geq 1$.
Then, by (\ref{expz}), 
\[
\widehat{z}_k-\widehat{z}_j
=(a^{(\nu)}_k(x_j, z_j)-a^{(\nu)}_j(x_j, z_j))\cdot w_{j}^{\nu}+O(w_{j}^{\nu+1})+O(w_{j}^{n+1}). 
\]
When the critical case $\nu=n$,  
as repeating same procedure for the expansion of $a^{(n)}_k(x_j, z_j)-a^{(n)}_j(x_j, z_j)$, 
it follows that 
$v_{n,m}(C,S,X; \{\widehat{z_{j}}\}) =v_{n,m}(C,S,X; \{z_{j}\}) =0$. 
The other cases can be proved similarly.
\end{proof}

In the above, under the jet extension property, we showed the well-definedness of obstruction classes. 
However, the authors do not know in general case.

\begin{problem}
Is there an example of $(C,S,X)$ for which the classes $u_{n,m}$ and $v_{n,m}$ actually depend on 
the choice of the system $\{w_{j}\}$ and the extension $\{z_{j}\}$, especially, 
in which the jet extension property fails ? 
\end{problem}

\subsection{Proof of jet extension property}\label{section:jetextproperty}
Let us go back to the situation where $L$ is a leaf of a Levi-flat hypersurface 
$(M,\mathcal{F})$ in $X$. 
Here we use the notation $L$ instead of $S$ in the subsections above.
As shrinking $X$ together with $L$ and $M$,  
we may assume that a connected component of $L$ is embedded in $X$.  
In \S4 below, it will be assumed only a small tubular neighborhood in $X$ of an elliptic curve $C$ contained in $L$.

\begin{proposition}\label{prop:jet_ext_welldef}
Let $(M,\mathcal{F})$, $L$ and $X$ be as above.  
Assume that (i) the holonomy group $\mathcal{H}(L)$ along $L$ is generated by $C^2$-flat diffeomorphisms 
and  
(ii) there exists a $C^\infty$ retraction $p\colon M\to L$
which is leafwise holomorphic. 
Then, the pair $(L,X)$ satisfies \textit{the jet extension property}.
\end{proposition}

\begin{proof}
Since the holonomy along $L$ is $C^2$-flat, 
there exist the data $\{(W_{j},w_{j})\}$ and $\{(\mathcal{U}_{j},u_{j})\}$ 
which satisfy the condition (ii),(iii) and (iv) 
described in \S \ref{section:Barrett-Fornaess}. 
Take a holomorphic function $f\in H^0(L, \mathcal{O}_L)$ and 
denote the restriction $f|_{V_j}$  by $f_j$ . 
From Lemma \ref{lem:f_tilde} below, for $n\geq1$ and each $j$,
we can take a holomorphic function 
$\widetilde{f}^{(n)}_j\colon W_j \to \mathbb{C}$  
such that 
$\widetilde{f}^{(n)}_j|_{V_{j}}=f_j$ and 
$\widetilde{f}^{(n)}_j|_{\mathcal{U}_j}-(p^*f)|_{\mathcal{U}_j}=O(u_j^n)$. 
Let 
\[
\widetilde{f}^{(n)}_{k}-\widetilde{f}^{(n)}_{j}=
\widetilde{a}^{(\nu)}_{jk}\cdot w_j^\nu+O(w_j^{\nu+1}) 
\]
be the expansion on $W_{jk}$, 
where $\widetilde{a}^{(\nu)}_{jk}$ is the pull back of a holomorphic function $a^{(\nu)}_{jk}$ 
on $V_{j}$ by the coordinate projection $W_{j} \to V_{j}$. 
On $\mathcal{U}_{jk}$, the expansion in $u_{j}$ is obtained by
\begin{eqnarray}
(\widetilde{f}^{(n)}_k-\widetilde{f}^{(n)}_j)|_{\mathcal{U}_{jk}}&=&
\widetilde{a}^{(\nu)}_{jk}|_{\mathcal{U}_{jk}}\cdot w_j^\nu|_{\mathcal{U}_{jk}}+O(u_j^{\nu+1})\nonumber \\
&=& (\pi_j^*a^{(\nu)}_{jk}+O(u_j))\cdot (u_j+O(u_j^2))^\nu+O(u_j^{\nu+1})\nonumber \\
&=& \pi_j^*a^{(\nu)}_{jk}\cdot u_j^\nu+O(u_j^{\nu+1}), \nonumber
\end{eqnarray}
where $\pi_j\colon \mathcal{U}_j\to V_j$ is the coordinate projection. 
On the other hand, we have 
\[
(\widetilde{f}^{(n)}_k-\widetilde{f}^{(n)}_j)|_{\mathcal{U}_{jk}}
=(\widetilde{f}^{(n)}_k-p^*f)|_{\mathcal{U}_{jk}}-(\widetilde{f}^{(n)}_j-p^*f)|_{\mathcal{U}_{jk}}
=O(u_j^n). 
\]
Comparing these expansions, we obtain $\nu\geq n$, so that the proposition follows. 
\end{proof}

\begin{lemma}\label{lem:f_tilde}
Let $\{W_j\}, \{w_j\}, \{u_j\}$ and $\{f_j\}$ be as above. 
Then, for $n\geq1$ and each $j$, 
there exist holomorphic functions 
$\widetilde{f}^{(n)}_j\colon W_j \to \mathbb{C}$ such that 
\[
{\widetilde{f}^{(n)}_j|_{V_{j}}=f_j \quad\textrm{and}\quad}
\widetilde{f}^{(n)}_j|_{\mathcal{U}_j}-(p^*f)|_{\mathcal{U}_j}=O(u_j^n).
\]  
\end{lemma}

\begin{proof}
The proof is by induction on $n$. 
We may assume that $W_j$ is Stein by shrinking if necessary. 
The statement is clear when $n=1$. 
Assume that the lemma is true for $n$. 
Let 
\[
\widetilde{f}^{(n)}_j|_{\mathcal{U}_j}-(p^*f)|_{\mathcal{U}_j}=\theta^{(n)}_j\cdot u_j^n+O(u_j^{n+1})
\]
be the expansion in $u_j$. 
As $u_j$ is constant along the leaves, 
it follows that $\overline{\partial}u_j\equiv 0$ holds on each leaves. 
Since $\widetilde{f}^{(n)}_j|_{\mathcal{U}_j}$ and $(p^*f)|_{\mathcal{U}_j}$ 
are leafwise holomorphic, 
$\theta^{(n)}_j$ is holomorphic on $V_j$. 
As taking holomorphic functions $\widetilde{\theta}^{(n)}_j \colon W_j \to \mathbb{C}$ 
such that $\widetilde{\theta}^{(n)}_j|_{V_j}={\theta}^{(n)}_j$, 
we define holomorphic functions $\widetilde{f}^{(n+1)}_j\colon W_j \to \mathbb{C}$ by 
\[
\widetilde{f}^{(n+1)}_j:=\widetilde{f}^{(n)}_j-\widetilde{\theta}^{(n)}_j\cdot w_j^n. 
\]
It follows by the inductive assumption and the choice of the systems $\{(W_{j},w_j)\}, \{(\mathcal{U}_j,u_j)\}$ that 
$\widetilde{f}^{(n+1)}_j|_{\mathcal{U}_j}-(p^*f)|_{\mathcal{U}_j}=O(u_j^{n+1})$ holds. 
In fact, we have$\vspace{2mm}$ 
\\
\centerline{
$
\begin{array}{lll}
\widetilde{f}^{(n+1)}_j|_{\mathcal{U}_j}-(p^*f)|_{\mathcal{U}_j}&=&
\widetilde{f}^{(n)}_j|_{\mathcal{U}_j}-(p^*f)|_{\mathcal{U}_j}-\widetilde{\theta}^{(n)}_j|_{\mathcal{U}_j} \cdot w_j^n|_{\mathcal{U}_j}
\vspace{2mm}
\\
&=&\theta^{(n)}_j(z_{j})\cdot u_j^n+O(u_j^{n+1})-(\theta^{(n)}_j(z_{j})+O(u_{j}))\cdot (u_{j}+O(u_{j}^2))^n
\vspace{2mm}
\\
&=&O(u_{j}^{n+1}).
\vspace{2mm}
\end{array}
$}
\\
Note that, we need at least $C^2$-flatness for the last estimate.
\end{proof}

\subsection{Proof of Theorem \ref{thm:ueda_nbhd}}\label{pfthm:ueda_nbhd} 
Take open coverings $\{U_{j}\}$ of $C$, $\{V_{j}\}$ of $V$ and $\{W_{j}\}$ of $W$ as beginning at \S3. 
Denote the transition functions of the normal bundles $N_{S/X}$ and $N_{C/S}$ by $t_{jk}$ and $s_{jk}$.
In this section, let us prove Theorem \ref{thm:ueda_nbhd} in a slightly general setting:
not only for the case ${\rm div}(f)=C$, 
but also the case $\{f=0\}=C$. 
Set $\lambda:={\rm mult}_Cf$ and $y_j:=(f^{1/\lambda})|_{V_j}$. 
Thus $\{y_{j}\}$ satisfies the setting as beginning at \S3, that is, $s_{jk}y_{k}=y_{j}$. 
In particular, $N_{C/S}$ is torsion.
In what follows, we prove this theorem under the assumption 
``(iii)' ${\rm type}(C, S, X)=\infty$ and $\{y_j\}$ is of extension type infinity" (Note that (iii)' is a milder condition than (iii)). 
First we prepare the following lemma.
The authors recomend the reader to skip the proof of Lemma 3.12 at first, 
since the proof is almost the same as (and much simpler than) that of Lemma \ref{lem:Ueda_codim2_key}.
\begin{lemma}
There exists an extension of $\{y_j\}$ of type $(2, 0)$. 
\end{lemma}

\begin{proof}
Take an extension $\{z_j\}$ of type $(1, 0)$ and denote by 
\[
      s_{jk}z_k = z_j+\sum_{\nu=1}^\infty\sum_{\mu=0}^\infty q_{jk}^{(\nu, \mu)}(x_j)\cdot w_j^{\nu}z_j^\mu 
\]
the expansion of $s_{jk}z_k$ in the variables $z_j$ and $w_j$. 
We will construct an extension $\{\zeta_j\}$ of $\{y_j\}$ of type $(2, 0)$ as the solution of a suitably determined Schr\"oder type functional equation 
\begin{equation}\label{eq:func_eq_for_q_1_m}
      z_j  =\zeta_j+\sum_{\mu=0}^\infty Q_j^{(1, \mu)}(x_j)\cdot w_j\zeta_j^\mu. 
\end{equation}
Let us explain how to choose the coefficient functions $\{Q_j^{(1, \mu)}\}$. 
First, take $\{Q_j^{(1, 0)}\}$ so that 
$Q_j^{(1, 0)}-t_{jk}^{-1}s_{jk}Q_k^{(1, 0)}=q_{jk}^{(1, 0)}$ holds for each $j$ and $k$ 
(Recall the condition that $v_{1, 0}\in H^1(C, N_{S/X}|_C^{-1}\otimes N_{C/S})$ vanishes). 
Next, 
we explain how to take $\{Q_j^{(1, m)}\}$ by assuming there exists $\{Q_j^{(1, \mu)}\}$ for each $\mu<m$ which satisfies the inductive assumption {\bf (condition)}$_{m}$: for any choice of the remaining coefficient functions $\{Q_j^{(1, \mu)}\}_{\mu\geq m}$, the solution $\{\zeta_j\}$ of the above Schr\"oder type functional equation is an extension of $\{y_j\}$ of type $(1, m)$ if exists. 
Consider the solution $\{\xi_j\}$ of the functional equation 
\[
      z_j  =\xi_j+\sum_{\mu=0}^{m-1} Q_j^{(1, \mu)}(x_j)\cdot w_j\xi_j^\mu, 
\]
which is an extension of type $(1, m)$ by the inductive assumption {\bf (condition)}$_{m}$. 
By considering the two-fold manner of expanding $s_{jk}\xi_k$ in the variables $\xi_j$ and $w_j$ (for the details, see the arguments below on comparing the equations (\ref{eq:funceq_1}) and (\ref{eq:funceq_2}), or in the proof of Lemma \ref{lem:Ueda_codim2_key}), we obtain the equality 
\begin{eqnarray}
s_{jk}\xi_k-\xi_j&=&-\sum_{\mu=0}^{m-1} 
\Big{(} t_{jk}^{-1}s_{jk}^{-\mu+1}Q_k^{(1, \mu)}(x_k(x_j))+h_{1, jk, 1\mu}(x_j)\Big{)}\cdot w_j\xi_j^\mu
-h_{1, jk, 1m}(x_j)\cdot w_j\xi_j^m\nonumber \\
&&+\sum_{\mu=0}^{m-1} \Big{(} Q_j^{(1, \mu)}(x_j)+q_{jk}^{(1, \mu)}(x_j)\Big{)}\cdot w_j\xi_j^\mu
+q_{jk}^{(1, m)}(x_j)\cdot w_j \xi_j^m+O(\xi_j^{m+1})\cdot w_j +O(w_j^2), \nonumber 
\end{eqnarray}
where $h_{1, jk, 1\ell}:=\textstyle\sum_{\mu=0}^{\ell-1}t_{jk}^{-1}s_{jk}^{-\mu+1}Q_{kj, 0(\ell-\mu)}^{(1, \mu)}(x_j)$ is the function determined by the coefficients $Q_{kj, 0q}^{(1, \mu)}$ of the expansion 
\[
Q_k^{(1, \mu)}(x_k)=Q_k^{(1, \mu)}(x_k(x_j, z_j, w_j))=Q_k^{(1, \mu)}(x_k(x_j, 0, 0))+\sum_{q=1}^\infty Q_{kj, 0q}^{(1, \mu)}(x_j)\cdot z_j^q+O(w_j). 
\]
Note that it follows from {\bf (condition)}$_{m}$ that 
\[
t_{jk}^{-1}s_{jk}^{-\mu+1}Q_k^{(1, \mu)}(x_k(x_j))-Q_j^{(1, \mu)}(x_j)=-h_{1, jk, 1\mu}(x_j)+q_{jk}^{(1, \mu)}(x_j)
\]
holds on $U_{jk}$ for each $\mu<m$. Take $\{Q_j^{(1, m)}\}$ so that 
\[
t_{jk}^{-1}s_{jk}^{-m+1}Q_k^{(1, m)}(x_k(x_j))-Q_j^{(1, m)}(x_j)=-h_{1, jk, 1m}(x_j)+q_{jk}^{(1, m)}(x_j)
\]
holds on $U_{jk}$ (here we used the condition that $v_{1, m}\in H^1(C, N_{S/X}|_C^{-1}\otimes N_{C/S}^{-m+1})$ vanishes). 
Then it is easily checked that {\bf (condition)}$_{m+1}$ holds. 

Finally, by {\bf (condition)}$_{m}$'s and the inverse function theorem, it is sufficient to show that the right hand side of the functional equation (\ref{eq:func_eq_for_q_1_m}) has a positive radius of convergence as a (formal) power series. 
By the same (or much easier) argument as in the proof of Lemma \ref{lem:Ueda_codim2_key} below, we obtain that 
\[
\zeta_j+\frac{2KRM}{1-(2K+1)R\cdot \zeta_j}\cdot w_j
=\zeta_j+2KRM\cdot w_j+2KR^2M(2K+1)\cdot w_j\zeta_j+2KR^3M(2K+1)^2\cdot w_j\zeta_j^2+\cdots
\]
is a dominant convergence sequence of the right hand side of the functional equation (\ref{eq:func_eq_for_q_1_m}) for sufficiently large positive numbers $K, R$, and $M$, which proves the lemma (see also the proof of \cite[Lemma 5 (i')]{K}). 
\end{proof}

Take an extension $\{z_j\}$ of $\{y_j\}$ of type $(2, 0)$ and a system $\{w_j\}$ of type $(1, 0)$. 
Let 
\begin{equation}\label{eq:eq_1}
\left(
    \begin{array}{c}
      t_{jk}w_k  \\
      s_{jk}z_k 
    \end{array}
  \right)
=
\left(
    \begin{array}{c}
       w_j\\
       z_j
    \end{array}
  \right)
+\sum_{\nu=2}^\infty\sum_{\mu=0}^\infty 
\left(
    \begin{array}{c}
       g_{jk}^{(\nu, \mu)}(x_j) \\
       q_{jk}^{(\nu, \mu)}(x_j)
    \end{array}
  \right)
\cdot w_j^{\nu}z_j^\mu 
\end{equation}
be the expansions on $W_{jk}$. 
We will refine the system $\{w_{j}\}$ and 
the extension $\{z_{j}\}$ of $\{y_{j}\}$ on $W_{j}$ \textit{keeping fixed $\{x_{j}\}$ and $\{y_{j}\}$}. 

The goal of this proof is 
to construct a new system $\{v_j\}$ and a new extension $\{\zeta_j\}$ of $\{y_j\}$ 
satisfying 
\begin{equation}\label{eq:eq_2}
\left(
    \begin{array}{c}
      t_{jk}v_k  \\
      s_{jk}\zeta_k 
    \end{array}
  \right)
=
\left(
    \begin{array}{c}
       v_j\\
       \zeta_j
    \end{array}
  \right)
\end{equation}
on each $W_{jk}$ 
after shrinking $W$ if necessary. 
Then $\{v_j\}$ gives a $U(1)$-flat structure on $\mathcal{O}_X(S)$ 
and $\{\zeta_j\}$ defines a complex hypersurface $Y$ of $W$ 
which intersects $S$ transversally along $C$. 
Let us construct such defining functions $\{v_j\}$ and $\{\zeta_j\}$ 
by solving a Schr\"oder type functional equation 
defined by (\ref{eq:eq_1}) and (\ref{eq:eq_2}) 
as in \cite[\S 4.2]{U} and \cite[\S 4.1]{K}. However, in our situation, 
some difficulties arise from non-compactness of $V$ and $Y$.  
In order to avoid the difficulties, we solve two functional equations together, 
whose solutions have a compact common zero $C=V\cap Y$. 

Let us consider the Schr\"oder type functional equation 
\begin{equation}\label{eq:main_funceq}
\left(
    \begin{array}{c}
      w_j  \\
      z_j  
    \end{array}
  \right)
=
\left(
    \begin{array}{c}
       v_j\\
       \zeta_j
    \end{array}
  \right)
+\sum_{\nu=2}^\infty\sum_{\mu=0}^\infty 
\left(
    \begin{array}{c}
       G_{j}^{(\nu, \mu)}(x_j) \\
       Q_{j}^{(\nu, \mu)}(x_j)
    \end{array}
  \right)
\cdot  v_j^{\nu}\zeta_j^\mu
\end{equation}
on each $W_j$, where the coefficient functions $G_{j}^{(\nu, \mu)}$ and $Q_{j}^{(\nu, \mu)}$ are holomorphic functions depending only on $x_j$. 
We will construct the coefficients $\{(G_{j}^{(\nu, \mu)}, Q_{j}^{(\nu, \mu)})\}$  
so that the solutions $\{( v_j, \zeta_j)\}$ exist and satisfy the equation (\ref{eq:eq_2}). 

In order to explain how to construct the coefficients, let us first observe the properties of them by assuming that $\{v_j\}$ and $\{\zeta_j\}$ are solutions of the equation (\ref{eq:eq_2}) on each $W_{jk}$. 
Consider the following two manners of the expansion of $(t_{jk}w_k, s_{jk}z_k)|_{W_{jk}}$ in $v_j$ and $\zeta_j$. 
The first one is obtained as follows: 
\begin{eqnarray}\label{eq:funceq_1}
\left(
    \begin{array}{c}
      t_{jk}w_k  \\
      s_{jk}z_k  
    \end{array}
  \right)
&=& 
\left(
    \begin{array}{c}
       t_{jk}v_k\\
       s_{jk}\zeta_k
    \end{array}
  \right)
+\sum_{(\nu, \mu)\geq(2, 0)}
\left(
    \begin{array}{c}
       t_{jk}G_{k}^{(\nu, \mu)}(x_k) \\
       s_{jk}Q_{k}^{(\nu, \mu)}(x_k)
    \end{array}
  \right)
\cdot  v_k^{\nu}\zeta_k^\mu
\\
&=& 
\left(
    \begin{array}{c}
       v_j\\
       \zeta_j
    \end{array}
  \right)
+\sum_{(\nu, \mu)\geq(2, 0)}
\left(
    \begin{array}{c}
       t_{jk}^{-\nu+1}s_{jk}^{-\mu}G_{k}^{(\nu, \mu)}(x_k) \\
       t_{jk}^{-\nu}s_{jk}^{-\mu+1}Q_{k}^{(\nu, \mu)}(x_k)
    \end{array}
  \right)
\cdot  v_j^{\nu}\zeta_j^\mu
\nonumber\\
&=& 
\left(
    \begin{array}{c}
       v_j\\
       \zeta_j
    \end{array}
  \right)
+\sum_{(\nu, \mu)\geq(2, 0)}
\left(
    \begin{array}{c}
       t_{jk}^{-\nu+1}s_{jk}^{-\mu}G_{k}^{(\nu, \mu)}(x_k(x_j))+h^{(1)}_{1,jk,\nu\mu}(x_j) \\
       t_{jk}^{-\nu}s_{jk}^{-\mu+1}Q_{k}^{(\nu, \mu)}(x_k(x_j))+h^{(2)}_{1,jk,\nu\mu}(x_j)
    \end{array}
  \right)
\cdot  v_j^{\nu}\zeta_j^\mu. 
\nonumber
\end{eqnarray}
Here the expansions of $(G_{k}^{(\nu, \mu)},Q_{k}^{(\nu, \mu)})$ in $z_j$ and $w_j$ are denoted by 
\[
\left(
    \begin{array}{c}
       G_{k}^{(\nu, \mu)}(x_k(x_j, z_j, w_j)) \\
       Q_{k}^{(\nu, \mu)}(x_k(x_j, z_j, w_j))
    \end{array}
  \right)
=
\left(
    \begin{array}{c}
       G_{k}^{(\nu, \mu)}(x_k(x_j)) \\
       Q_{k}^{(\nu, \mu)}(x_k(x_j))
    \end{array}
  \right)
+\sum_{(p, q)\geq(0,1)}
\left(
    \begin{array}{c}
       G_{kj, pq}^{(\nu, \mu)}(x_j) \\
       Q_{kj, pq}^{(\nu, \mu)}(x_j)
    \end{array}
  \right)
\cdot w_j^pz_j^q
\]
and 
$h^{(i)}_{1,jk,\nu\mu}(x_j)$'s are the coefficients of $v_j^\nu\zeta_j^\mu$ in the expansion of
\[
\sum_{(c, d)\geq(2, 0)}\sum_{(p, q)\geq(0,1)}\left(
    \begin{array}{c}
       t_{jk}^{-\nu+1}s_{jk}^{-\mu}G_{kj, pq}^{(\nu, \mu)}(x_j) \\
       t_{jk}^{-\nu}s_{jk}^{-\mu+1}Q_{kj, pq}^{(\nu, \mu)}(x_j)
    \end{array}
  \right)
\cdot\left(v_j+\sum G_j^{(a, b)}v_j^a\zeta_j^b\right)^p
\cdot\left(\zeta_j+\sum Q_j^{(a, b)}v_j^a\zeta_j^b\right)^q
\cdot  v_j^{c}\zeta_j^{d}. 
\]

The second one is as follows: 
\begin{eqnarray}\label{eq:funceq_2}
\left(
    \begin{array}{c}
      t_{jk}w_k  \\
      s_{jk}z_k 
    \end{array}
  \right)
&=&
\left(
    \begin{array}{c}
       w_j\\
       z_j
    \end{array}
  \right)
+\sum_{(\nu, \mu)\geq(2,0)}
\left(
    \begin{array}{c}
       g_{jk}^{(\nu, \mu)}(x_j) \\
       q_{jk}^{(\nu, \mu)}(x_j)
    \end{array}
  \right)
\cdot w_j^{\nu}z_j^\mu 
\\
&=&
\left(
    \begin{array}{c}
       v_j\\
       \zeta_j
    \end{array}
  \right)
+\sum_{(\nu, \mu)\geq(2,0)}
\left(
    \begin{array}{c}
       G_{j}^{(\nu, \mu)}(x_j)+h^{(1)}_{2,jk,\nu\mu} \\
       Q_{j}^{(\nu, \mu)}(x_j)+h^{(2)}_{2,jk,\nu\mu} 
    \end{array}
  \right)
\cdot  v_j^{\nu}\zeta_j^\mu.
\nonumber
\end{eqnarray}
Here $h^{(i)}_{2,jk,\nu\mu}(x_j)$'s are the coefficients of $v_j^\nu\zeta_j^\mu$ 
in the expansion of 
\[
\sum_{(\nu, \mu)\geq(2,0)}\left(
    \begin{array}{c}
       g_{jk}^{(\nu, \mu)}(x_j) \\
       q_{jk}^{(\nu, \mu)}(x_j)
    \end{array}
  \right)
\cdot\left(v_j+\sum G_j^{(a, b)}v_j^a\zeta_j^b\right)^\nu\cdot\left(\zeta_j+\sum Q_j^{(a, b)}v_j^a\zeta_j^b\right)^\mu. 
\]

By comparing two expansions (\ref{eq:funceq_1}) and (\ref{eq:funceq_2}), 
it is observed that one should take $\{G_j^{(\nu, \mu)}\}$ as a solution of the $\delta$-equation 
\begin{equation}\label{eq:delta-eq_1}
\delta\{(U_j, G_j^{(\nu, \mu)})\}=\{(U_{jk}, h^{(1)}_{1,jk,\nu\mu}-h^{(1)}_{2,jk,\nu\mu})\}
\in \check{Z}^1(\{U_j\}, N_{S/X}|_C^{-\nu+1}\otimes N_{C/S}^{-\mu})
\end{equation}
and $\{Q_j^{(\nu, \mu)}\}$ as a solution of 
\begin{equation}\label{eq:delta-eq_2}
\delta\{(U_j, Q_j^{(\nu, \mu)})\}=\{(U_{jk}, h^{(2)}_{1,jk,\nu\mu}-h^{(2)}_{2,jk,\nu\mu})\}
\in \check{Z}^1(\{U_j\}, N_{S/X}|_C^{-\nu}\otimes N_{C/S}^{-\mu+1}). 
\end{equation}
In fact, the following lemma tells that 
$\{G_j^{(\nu, \mu)}\}, \{Q_j^{(\nu, \mu)}\}$ are inductively determined by 
(\ref{eq:delta-eq_1}),(\ref{eq:delta-eq_2}) 
with some estimates to assure the existence of solutions 
of (\ref{eq:main_funceq}). 
Further, it is not difficult to check that, for $(n,m)\geq(2,0)$,  
$h^{(i)}_{1,jk,nm}(x_j), h^{(i)}_{2,jk,nm}(x_j)$ depend on 
the choice of $\{G_j^{(\nu, \mu)}\}, \{Q_j^{(\nu, \mu)}\}$ only for 
$\nu<n$ and $\mu\leq m$, or $\nu\leq n$ and $\mu<m$. 

\begin{lemma}\label{lem:Ueda_codim2_key}
There exists a power series $A(X, Y)=\sum_{\nu=2}^\infty\sum_{\mu=0}^\infty A_{\nu, \mu}X^\nu Y^\mu$ with positive radius of convergence which satisfy the following property 
for each $(n, m)\geq (2, 0)$: \\
if $\{G_j^{(\nu, \mu)}\}$ and $\{Q_j^{(\nu, \mu)}\}$ satisfy 
{\bf (condition)}$_{(\nu, \mu)}$ for each $(2,0)\leq (\nu, \mu)<(n, m)$, 
then there exist $\{G_j^{(n, m)}\}$ and $\{Q_j^{(n, m)}\}$ with {\bf (condition)}$_{(n, m)}$.
\begin{description}
\item[(condition)$_{(\nu, \mu)}$] 
$\{G_j^{(\nu, \mu)}\}$ and $\{Q_j^{(\nu, \mu)}\}$ are solutions of the $\delta$-equations (\ref{eq:delta-eq_1}) and (\ref{eq:delta-eq_2}) with $\max_j\sup_{U_j}|G_j^{(\nu, \mu)}|\leq A_{\nu, \mu}$ and $\max_j\sup_{U_j}|Q_j^{(\nu, \mu)}|\leq A_{\nu, \mu}$ respectively.
\end{description} 
\end{lemma}

\begin{proof}
Let $K_{k, \ell}:=K(N_{S/X}|_C^{-k}\otimes N_{C/S}^{-\ell})$ be the constant 
as in \cite[Lemma 3]{U} (Kodaira-Spencer's lemma \cite{KS}) and 
set $K:=\max_{k, \ell} K_{k, \ell}$ (here we used the fact that $N_{S/X}|_C$ and $N_{C/S}$ are torsion). 
Take sufficiently large $M,R>0$ 
such that $\textstyle\sup_{W_{jk}}|t_{jk}w_k-w_j|<M$,  $\textstyle\sup_{W_{jk}}|s_{jk}z_k-z_j|<M$ and 
$\Delta_{R}=\{(x_j, z_j, w_j)\mid x_j\in U_{jk}, |z_j|<R^{-1}, |w_j|<R^{-1}\}\subset W_{jk}$ hold for each $j, k$ 
(Strictly speaking, we can not take such $R$ in general. 
In the general case, we have to consider a new open covering $\{U_j^*\}$ of $C$ such that each $U_j^*$ is a relatively compact subset of $U_j$. After replacing $R$ with a sufficiently large number, $\Delta_R$ with $\{(x_j, z_j, w_j)\mid x_j\in U_{j}\cap U_k^*, |z_j|<R^{-1}, |w_j|<R^{-1}\}$, and $M$ with $2M$, and so on, 
this difficulty can be avoided. See \cite[p. 599]{U} and \cite[Remark 2]{K} for the details). 
Consider the solution $A(X, Y)$ of the cubic equation
 \begin{equation}\label{eq:funceq_A}
  A(X, Y)=KR\cdot\left(\frac{A(X, Y)(Y+A(X, Y))}{1-RY-RA(X, Y)}+\frac{A(X, Y)(X+A(X, Y))+MR(X+A(X, Y)^2)}{(1-RX-RA(X, Y))(1-RY-RA(X, Y))}\right).   
 \end{equation}
Though the functional equation $(\ref{eq:funceq_A})$ has three solutions, the solution $A$ is uniquely determined by the condition that $A(X, Y)=O(X^2)$. 
In fact, the coefficients are inductively determined. 

Let us check that this $A(X, Y)$ satisfies the required property. 
Assume the existence of $\{G_j^{(\nu, \mu)}\}, \{Q_j^{(\nu, \mu)}\}$ 
with {\bf(condition)}$_{(\nu, \mu)}$ for each $(2,0)\leq (\nu, \mu)<(n, m)$. 
Then, by the assumption $(iii)$ of types,
the existence of solutions $\{G_j^{(n, m)}\}, \{Q_j^{(n, m)}\}$ 
satisfying {\bf(condition)}$_{(n, m)}$
can be shown as the following steps, even if $(n,m)=(2,0)$.  

\vskip3mm
\noindent{\bf Construction of $\{G_j^{(n, m)}\}$ and $\{Q_j^{(n, m)}\}$. }
\\
Consider the solution of the functional equation 
\[
\left(
    \begin{array}{c}
      w_j  \\
      z_j  
    \end{array}
  \right)
=
\left(
    \begin{array}{c}
       v_j\\
       \zeta_j
    \end{array}
  \right)
+\sum_{(2, 0)\leq (\nu, \mu)<(n, m)}
\left(
    \begin{array}{c}
       G_{j}^{(\nu, \mu)}(x_j) \\
       Q_{j}^{(\nu, \mu)}(x_j)
    \end{array}
  \right)
\cdot  v_j^{\nu}\zeta_j^\mu. 
\]
From the $\delta$-equations (\ref{eq:delta-eq_1}) and (\ref{eq:delta-eq_2}) for $(2,0)\leq (\nu, \mu)<(n, m)$, 
it follows from the inductive argument that the system $\{v_j\}$ is of type $(n-1,m)$ and the extension $\{\zeta_j\}$ is of type $(n,m)$. 
By using this fact, we have the two expansions of $(t_{jk}w_{k},s_{jk}z_{k})$ (see also \cite[Lemma 7]{K}):
\begin{eqnarray}
\left(
    \begin{array}{c}
      t_{jk}w_k  \\
      s_{jk}z_k  
    \end{array}
  \right)
&=& 
\left(
    \begin{array}{c}
       t_{jk}v_k\\
       s_{jk}\zeta_k
    \end{array}
  \right)
+\sum_{(2, 0)\leq (\nu, \mu)<(n, m)}
\left(
    \begin{array}{c}
       t_{jk}G_{k}^{(\nu, \mu)}(x_k) \\
       s_{jk}Q_{k}^{(\nu, \mu)}(x_k)
    \end{array}
  \right)
\cdot  v_k^{\nu}\zeta_k^\mu
\nonumber
\\
&=& 
\left(
    \begin{array}{c}
       t_{jk}v_k\\
       s_{jk}\zeta_k
    \end{array}
  \right)
+\sum_{(2, 0)\leq (\nu, \mu)<(n, m)}
\left(
    \begin{array}{c}
       t_{jk}^{-\nu+1}s_{jk}^{-\mu}G_{k}^{(\nu, \mu)}(x_k) \\
       t_{jk}^{-\nu}s_{jk}^{-\mu+1}Q_{k}^{(\nu, \mu)}(x_k)
    \end{array}
  \right)
\cdot  v_j^{\nu}\zeta_j^\mu+O(v_j^{n+1})
\nonumber\\
&=& 
\left(
    \begin{array}{c}
       t_{jk}v_k\\
       s_{jk}\zeta_k
    \end{array}
  \right)
+\sum_{(2, 0)\leq (\nu, \mu)<(n, m)}
\left(
    \begin{array}{c}
       t_{jk}^{-\nu+1}s_{jk}^{-\mu}G_{k}^{(\nu, \mu)}(x_k(x_j))+h^{(1)}_{1,jk,\nu\mu}(x_j) \\
       t_{jk}^{-\nu}s_{jk}^{-\mu+1}Q_{k}^{(\nu, \mu)}(x_k(x_j))+h^{(2)}_{1,jk,\nu\mu}(x_j)
    \end{array}
  \right)
\cdot  v_j^{\nu}\zeta_j^\mu
\nonumber \\
&&+
\left(
    \begin{array}{c}
       h^{(1)}_{1,jk,nm}(x_j) \\
       h^{(2)}_{1,jk,nm}(x_j)
    \end{array}
  \right)
\cdot  v_j^{n}\zeta_j^m +O(\zeta_j^{m+1})\cdot v_j^{n}+O(v_j^{n+1})
\nonumber
\end{eqnarray}
and 
\begin{eqnarray}
\left(
    \begin{array}{c}
      t_{jk}w_k  \\
      s_{jk}z_k 
    \end{array}
  \right)
&=&
\left(
    \begin{array}{c}
       w_j\\
       z_j
    \end{array}
  \right)
+\sum_{(2, 0)\leq (\nu, \mu)\leq(n, m)}
\left(
    \begin{array}{c}
       g_{jk}^{(\nu, \mu)}(x_j) \\
       q_{jk}^{(\nu, \mu)}(x_j)
    \end{array}
  \right)
\cdot w_j^{\nu}z_j^\mu +O(\zeta_j^{m+1})\cdot v_j^{n}+O(v_j^{n+1})
\nonumber 
\\
&=&
\left(
    \begin{array}{c}
       v_j\\
       \zeta_j
    \end{array}
  \right)
+\sum_{(2, 0)\leq (\nu, \mu)<(n, m)}
\left(
    \begin{array}{c}
       G_{j}^{(\nu, \mu)}(x_j)+h^{(1)}_{2,jk,\nu\mu} \\
       Q_{j}^{(\nu, \mu)}(x_j)+h^{(2)}_{2,jk,\nu\mu} 
    \end{array}
  \right)
\cdot  v_j^{\nu}\zeta_j^\mu\nonumber \\
&&+
\left(
    \begin{array}{c}
       h^{(1)}_{2,jk,nm}(x_j) \\
       h^{(2)}_{2,jk,nm}(x_j)
    \end{array}
  \right)
\cdot  v_j^{n}\zeta_j^m +O(\zeta_j^{m+1})\cdot v_j^{n}+O(v_j^{n+1}).
\nonumber
\end{eqnarray}
As $\{G_j^{(\nu, \mu)}\}$ and $\{Q_j^{(\nu, \mu)}\}$ are solutions of the $\delta$-equations (\ref{eq:delta-eq_1}) and (\ref{eq:delta-eq_2}) for each $(\nu, \mu)<(n, m)$, we have  
\[
\left(
    \begin{array}{c}
      t_{jk}v_k  \\
      s_{jk}\zeta_k  
    \end{array}
  \right)
=
\left(
    \begin{array}{c}
       v_j\\
       \zeta_j
    \end{array}
  \right)-
\left(
    \begin{array}{c}
       h^{(1)}_{1,jk,nm}-h^{(1)}_{2,jk,nm} \\
       h^{(2)}_{1,jk,nm}-h^{(2)}_{2,jk,nm}
    \end{array}
  \right)
\cdot  v_j^{n}\zeta_j^m
+O(\zeta_j^{m+1})\cdot v_j^{n}+O(v_j^{n+1}).  
\] 
From the assumption $(iii)$, 
$u_{n-1,m}(C, S, X;\{v_{j}\})=0$ and $v_{n,m}(C,S,X;\{\zeta_j\})=0$ hold. 
Thus we obtain the solutions $\{G_j^{(n, m)}\}, \{Q_j^{(n, m)}\}$ 
of (\ref{eq:delta-eq_1}), (\ref{eq:delta-eq_2}). 
\vskip3mm
\noindent{\bf Estimate of $\{G_j^{(n, m)}\}$ and $\{Q_j^{(n, m)}\}$. }
\\
Let us estimate for such $\{G_j^{(n, m)}\}$ and $\{Q_j^{(n, m)}\}$. 
From the definition of $h^{(i)}_{1,jk,\nu\mu}(x_j)$ and 
the Cauchy's estimate for $G_{k}^{(\nu, \mu)}$ and $Q_{k}^{(\nu, \mu)}$ 
on $\Delta_{R,x_j}=\{|z_j|<R^{-1}, |w_j|<R^{-1}\}$,
it follows that $|h^{(i)}_{1,jk,nm}(x_j)|$ is bounded by the coefficient of $X^nY^m$ 
in the expansion of
\begin{eqnarray}
&&\sum_{(\nu, \mu)\geq(2,0)}\sum_{(p, q)\geq(0,1)}(A_{\nu, \mu}R^{p+q})
\cdot(X+A(X, Y))^p\cdot (Y+A(X, Y))^q
\cdot  X^{\nu}Y^\mu\nonumber \\
&=&\sum_{(\nu, \mu)\geq(2,0)}\sum_{q=1}^\infty (A_{\nu, \mu}R^{q})\cdot (Y+A(X, Y))^q
\cdot  X^{\nu}Y^\mu \nonumber \\
&+&\sum_{(\nu, \mu)\geq(2,0)}\sum_{(p, q)\geq(1,0)} (A_{\nu, \mu}R^{p+q})
\cdot(X+A(X, Y))^p\cdot (Y+A(X, Y))^q
\cdot  X^{\nu}Y^\mu \nonumber \\
&=&\frac{RA(X, Y)(Y+A(X, Y))}{1-RY-RA(X, Y)}+\frac{RA(X, Y)(X+A(X, Y))}{(1-RX-RA(X, Y))(1-RY-RA(X, Y))}. \nonumber
\end{eqnarray}

Similarly, $|h^{(i)}_{2,jk,nm}(x_j)|$ is bounded by the coefficient of $X^nY^m$ in the expansion of
\begin{eqnarray}
&&\sum_{(\nu, \mu)\geq(2,0)}MR^{\nu+\mu}
\cdot(X+A(X, Y))^\nu\cdot (Y+A(X, Y))^\mu\nonumber \\
&=&\frac{MR^2(X+A(X, Y))^2}{(1-RX-RA(X, Y))(1-RY-RA(X, Y))}. \nonumber
\end{eqnarray}
Thus, from \cite[Lemma 3]{U} and the functional equation (\ref{eq:funceq_A}), 
we can choose the solutions $\{G_j^{(n, m)}\}$ and $\{Q_j^{(n, m)}\}$ of the $\delta$-equations 
(\ref{eq:delta-eq_1}) and (\ref{eq:delta-eq_2}) with 
$\max_j\sup_{U_j}|G_j^{(n, m)}|\leq A_{n, m}$ and $\max_j\sup_{U_j}|Q_j^{(n, m)}|\leq A_{n, m}$ respectively.
\end{proof}

Let $\{G^{(\nu, \mu)}_j\}$, $\{Q^{(\nu, \mu)}_j\}$ be 
the coefficients defined by Lemma \ref{lem:Ueda_codim2_key}. 
From the inverse function theorem, 
we obtain the solutions $\{v_j\}$ and $\{\zeta_j\}$ of the functional equation (\ref{eq:main_funceq}). 
Again by comparing two expansions of $(t_{jk}w_{k},s_{jk}z_{k})$ as 
(\ref{eq:funceq_1}) and (\ref{eq:funceq_2}) 
{\it without assuming} (\ref{eq:eq_2}), 
it follows that these solutions satisfy the equation (\ref{eq:eq_2}). 
The proof is completed.
\qed

\begin{remark}\label{rmk:case_E_1}
Let us denote by $\mathcal{E}_1(C)$ the set of all elements $F\in{\rm Pic}^0(C)$ which satisfies the condition $|\log d(\mathbb{I}_C, F^{n})|=O(\log n)$ as $n\to\infty$, 
where $\mathbb{I}_T\in{\rm Pic}^0(C)$ is the trivial line bundle and $d$ is an invariant distance of ${\rm Pic}^0(C)$ ($\mathcal{E}_1(C)$ does not depend on the choice of $d$, see \cite[\S 4.1]{U}). 
Theorem \ref{thm:ueda_nbhd} also holds when $N_{S/X}|_C\cong N_{C/S}$ and $N_{S/X}|_C, N_{C/S}\in \mathcal{E}_1(C)$, which gives a corrected version of \cite[Theorem $(ii)$]{K}. 
\end{remark}


\section{Proof of Theorem \ref{thm:main} and Theorem \ref{thm:main_higher_dim}}
\label{proof of theorem}

\subsection{Proof of Theorem \ref{thm:main_higher_dim}}
First we prove Theorem \ref{thm:main_higher_dim}. 
The outline of the proof is based on that of \cite[Theorem 1]{B}. 
We use the result on the $U(1)$-flatness in Theorem \ref{thm:ueda_nbhd} (we do not use the result on the existence of a transversal). 

We will apply Theorem \ref{thm:ueda_nbhd} 
instead of the original one \cite[Theorem 3]{U}. 
We lead to the contradiction 
by assuming that there exists a Levi-flat $C^\infty$-embedding $(M, \mathcal{F})\to X$. 
By shrinking $X$ to an open tubular neighborhood of $C$ if necessary, 
we may assume that $L$ is embedded in $X$.
First we will check the conditions to apply Theorem \ref{thm:ueda_nbhd}. 

\begin{lemma}\label{lem:key}
By shrinking $X$ to an open tubular neighborhood of $C$ if necessary, 
the followings hold: \\
$(a)$ $N_{L/X}$ is holomorophically trivial,  \\
$(b)$ the triple $(C, L, X)$ is of infinite type (in the sense of Definition \ref{def:infintype}), and \\
$(c)$ the holomorphic function $f$ in (iii) of Theorem \ref{thm:main_higher_dim} is of extension type infinity. 
\end{lemma}

\begin{proof}
Since the holonomy of the foliation $\mathcal{F}$ is $C^\infty$-flat along $L$, 
it follows that the normal bundle $N_{L/X}$ is holomorophically trivial and that there exists a system $\{w_j\}$ of type $n+1$ for each integer $n\geq 1$, see Proposition \ref{prop:BAppend}. 
Thus the obstruction class $u_{n, m}(C, L, X; \{w_j\})$ vanishes for each integer $n, m$. 
It follows from Lemma \ref{lemma_well-def_of_u_n_m} and Proposition \ref{prop:jet_ext_welldef} that $u_{n, m}(C, L, X)=0$ holds for {\it any} system of type $(n, m)$. 
Hence, the triple $(C, L, X)$ is of infinite type, see also Remark \ref{rmk:un_unm}. 
The assertion $(c)$ follows from Proposition \ref{prop:BAppend}, 
Lemma \ref{lem:weldef_v_n_m_z_j}, 
and Proposition \ref{prop:jet_ext_welldef}. 
\end{proof}

From Lemma \ref{lem:key} and Theorem \ref{thm:ueda_nbhd}, 
we may assume that $\mathcal{O}_X(L)$ is $U(1)$-flat, in particular, $t_{jk}\equiv 1$. 
Thus, we can take a global holomorphic defining function $h\colon X\to \mathbb{C}$ of $L$. 
More precisely, we can take $h$ as follows: 

\begin{lemma}\label{lem:function_h}
By shrinking $X$ to an open tubular neighborhood of $C$ if necessary, 
there exists a global holomorphic defining function $h\colon X\to \mathbb{C}$ of $L$ such that 
$d({\rm Re}\,h|_M)(x)\not=0$ holds for each point $x\in L$. 
\end{lemma}

\begin{proof}
Let $\{(W_j, w_j)\}$ be the system of type $r=2$ as in Proposition \ref{prop:BAppend}. 
In the last part of the proof of Theorem \ref{thm:ueda_nbhd}, 
we solved the functional equation 
\[
w_j=h+\sum_{\nu=2}^\infty\sum_{\mu=0}^\infty G_{j}^{(\nu, \mu)}\cdot h^\nu\zeta_j^\mu
\]
to construct a global function $h$. 
From this equation, we obtain that 
\[
d({\rm Re}\,h|_{W_j \cap M})(x)=d({\rm Re}\,w_j|_{W_j \cap M})(x)\not=0
\] 
holds for each $x\in L$. 
\end{proof}

Take a function $h$ as in Lemma \ref{lem:function_h}. 
Fix a leaf $L'\in\mathcal{F}$ which accumulates on $L$ and set $A:=p^{-1}(E)\cap L'$, where $E\subset C$ is an elliptic curve as in Theorem \ref{thm:main_higher_dim}. 
It follows from the holonomy condition that $A$ is diffeomorphic to an annulus 
which is accumulating on $E$. 
From Lemma \ref{lem:function_h}, ${\rm Re}\,h|_A$ is a {\it positive} harmonic function which tends to $0$ on one of the boundaries. 
Let $A$ be biholomorphic to the punctured disc. 
As compactfying $A$ to the unit disk and extending $h$, 
we can apply the maximal principle to $-{\rm Re}\,h|_A$ to lead to the contradiction \cite{B}. 
If $A$ is biholomorphic to an annulus, by considering the Laurent expansion, we can conclude that $h|_A\equiv0$ holds, 
which contradicts Lemma \ref{lem:function_h}. 
This completes the proof of Theorem \ref{thm:main_higher_dim}. 
\qed

\subsection{Proof of Theorem \ref{thm:main}}
When ${\rm div}(f)=C$, Theorem \ref{thm:main} directly follows from Theorem \ref{thm:main_higher_dim}. 
In general case, it is easily observed that $f$ defines an elliptic fibration containing $C$ as a singular fiber. 
By replacing $C$ with a general fiber, we can apply Theorem \ref{thm:main_higher_dim} to show the non-embeddability. 
\qed

\section{Some examples}\label{section:eg}
In this section, we give some examples. 
The following Example \ref{eg:cor_proof} shows Corollary \ref{cor:main}. 
\begin{example}\label{eg:reeb_times_C}
Let $(S^3, \mathcal{F}_{\rm Reeb}, J_{\mathcal{F}})$ be the Reeb foliation in Example \ref{eg:reeb} 
and $\Sigma$ a Riemann surface  (possibly non-compact). 
Consider a Levi-flat manifold $(S^3\times \Sigma, \mathcal{G}, J_{\mathcal{G}})$
where $\mathcal{G}=\{L\times \Sigma~|~L\in \mathcal{F}_{\rm Reeb}\}$. 
The map $\widetilde{R}\setminus\{z=0\}\ni(z, u)\mapsto z\in\mathbb{C}^*$ induces a retraction $p$ 
as in the condition (ii) of Theorem \ref{thm:main}.
The other conditions also obviously hold. 
Therefore, this Levi-flat manifold does not admit 
a $C^{\infty}$-embedding into any complex 3-manifolds. 
Note that this example is not covered 
by \cite{B}, \cite{BI} and \cite{D} if $\Sigma$ is non-compact. 
The authors do not know 
if $(S^3, \mathcal{F}_{\rm Reeb}, J_{\mathcal{F}})$ admits a Levi-flat $C^{\infty}$-embedding into a complex 3-manifold  or higher dimensional one. 
The above example showed that 
there exist no such embeddings of $(S^3, \mathcal{F}_{\rm Reeb}, J_{\mathcal{F}}) \times \mathbb{C}$ (or $\mathbb{D}$) into complex 3-manifolds. 
\end{example}

\begin{remark}
This also holds for turbulized Levi-flat manifolds instead of the Reeb foliation. 
The turbulization procedure for Levi-flat manifolds was introduced in \cite{HM}. 
\end{remark}


For our non-embeddability results, the assumption 
on the existence of a holomorphically embedded torus is essential. 
Actually, there exists a Levi-flat $5$-manifold 
which is homeomorphic to $(S^3, \mathcal{F}_{\rm Reeb}, J_{\mathcal{F}}) \times \mathbb{C}$ 
and which is embeddable in $\mathbb{C}^3$, see \cite[\S 9.1]{FT}. 
In this example, the torus is totally real and the holonomy is $C^{r}$-flat (not $C^{\infty}$-flat). 

\begin{example}
Let $(R,\mathcal{F},J_{\mathcal{F}})$ be the $5$-dimensional Reeb component 
described in Example \ref{eg:reeb}. 
The boundary leaf $L$ is a $2$-dimensional Hopf manifold 
and it contains an elliptic curve 
\[
C=(\mathbb{C}\times\{0\}\times \{0\}) \backslash \{(0',0)\}/\langle\lambda \rangle^{\mathbb{Z}}.
\] 
Consider the map $f\colon L\to \mathbb{C}P^{1}$ induced by the map 
$\mathbb{C}^2\backslash \{0'\} \ni (z_1, z_2)\mapsto [z_1, z_2]\in \mathbb{C}P^{1}$. 
Then $(z_2/z_1)\circ f$ is a global holomorphic defining function of $C$ on its neighborhood. 
It follows from Theorem \ref{thm:main} that the Reeb component $(R,\mathcal{F},J_{\mathcal{F}})$ does not admit 
a Levi-flat $C^{\infty}$-embedding into any complex $3$-manifolds.
Our result shows that a neighborhood of $C$ can not be embeddable (cf. \cite{D}).
\end{example}

\begin{example}\label{eg:cor_proof} Let $C$ be an elliptic curve and $L:=C\times\mathbb{C}$. 
Take a representation $\rho:\pi_1(L) \to \textrm{Diff}^{\infty}_{+}(\mathbb{R})$. 
Assume that
$\rho (\alpha)$ has a $C^{\infty}$-flat contracting fixed point at $0$ 
and $\rho (\beta)$ is the identity 
for generators $\alpha, \beta \in \pi_{1}(L)$. 
After usual suspension procedure by $\rho$, 
we obtain a Levi-flat manifold $(M_{1},\mathcal{F}_{1}, J_{\mathcal{F}_{1}})$
which is a foliated $\mathbb{R}$-bundle 
and is diffeomorphic to $T^2\times \mathbb{R}^2\times \mathbb{R}$. 
Obviously, this example satisfies the assumption in Theorem \ref{thm:main}. 
Therefore, this is Levi-flat $C^{\infty}$-nonembeddable into any complex 3-manifolds. 
On the other hand, in \cite[Example 1]{BI}, Barrett and Inaba gave an example 
$(N, \mathcal{G},J_{\mathcal{G}})$
which admits a Levi-flat $C^{\infty}$-embedding into $\mathbb{C}^{2}$ 
(Precisely speaking, we need to slightly modify it so as to be an embedding). 
In fact, 
the $C^{\infty}$ Levi-flat hypersurface in $\mathbb{C}^2$ is obtained by the map 
$\Psi:\{(z,t)\in \mathbb{C}\times \mathbb{R}~|~\textrm{Re}~z > 0,~0< \textrm{Im}~z < 1 \}\to \mathbb{C}^2$ such that
\[
  \Psi(z,t) = \begin{cases}
    (e^{2\pi\sqrt{-1}z},f^{-1}(f(t)+z)) & (t>0) \\
    (e^{2\pi\sqrt{-1}z},t) & (t\leq 0),
  \end{cases}
\]
where $f(t)=e^{\frac{1}{t}}$. The leaves consist of annuli and planes. 
It follows by the construction that the foliation has a $C^{\infty}$-flat (one sided) contracting holonomy 
along the punctured disk at $t=0$.
Then we obtain a Levi-flat $5$-manifold 
$(M_{2},\mathcal{F}_{2}, J_{\mathcal{F}_{2}}):=(N,\mathcal{G}, J_{\mathcal{G}})\times (\mathbb{C}^*,J_{st})$
which is Levi-flat $C^{\infty}$-embeddable in $\mathbb{C}^{3}$. 
By choosing a suitable representation $\rho$ in the above construction, 
we can get a non-embeddable example $(M_{1},\mathcal{F}_{1}, J_{\mathcal{F}_{1}})$
which is diffeomorphic to $(M_{2},\mathcal{F}_{2}, J_{\mathcal{F}_{2}})$ as $C^{\infty}$-foliated manifolds,  
so that the Corollary \ref{cor:main} follows. 
\end{example}
\vskip3mm
{\bf Acknowledgment. } 
The first author is supported by the Grant-in-Aid for Scientific Research (KAKENHI No.25-2869). 


\end{document}